\documentclass[12pt,reqno]{article}
\usepackage{amsmath}
\usepackage{amsfonts}
\usepackage{amsmath,amsfonts,amsthm,amssymb,color,hyperref}
\numberwithin{equation}{section}

\usepackage{graphicx}
\usepackage{comment}
\usepackage{mathtools}

\newtheorem{theorem}{Theorem}[section]
\newtheorem{proposition}[theorem]{Proposition}
\newtheorem{lemma}[theorem]{Lemma}
\newtheorem{corollary}[theorem]{Corollary}

\theoremstyle{definition}
\newtheorem{Remark}[theorem]{Remark}
\newtheorem*{remark}{Remark}

\newtheorem{definition}[theorem]{Definition}
\newtheorem{example}[theorem]{Example}

\newcommand{\dee}{{\partial_z}}
\newcommand{\deeb}{{\partial_{\overline z}}}

\newcommand\Z{\mathbf{Z}}
\newcommand\R{\mathbf{R}}
\newcommand\T{\mathbf{T}}
\newcommand\C{\mathbf{C}}

\newcommand\D{\mathbf{D}}
\renewcommand\P{\mathbf{P}}
\newcommand\E{\mathbf{E}}
\newcommand\complex{\mathbf{C}}
\newcommand\eps{\varepsilon}
\newcommand\zbar{{\bar{z}}}

\newcommand\supp{{\operatorname{supp}}}

\begin{document}

\title{  Homogenization of iterated singular integrals with applications to random quasiconformal maps}

\author{Kari Astala, Steffen Rohde, Eero Saksman, Terence Tao}

\date{}
\maketitle

\begin{abstract}

We study homogenization of iterated  randomized singular integrals and homeomorphic solutions to the Beltrami differential equation with a random Beltrami coefficient. More precisely, let
 $(F_j)_{j \geq 1}$ be a sequence of normalized homeomorphic solutions to the  planar Beltrami equation $\deeb F_j (z)=\mu_j(z,\omega) \dee F_j(z),$ where the random dilatation  satisfies
$|\mu_j|\leq k<1$ and has  locally periodic statistics,
for example of the type
\begin{equation}\label{eq:expl}
\mu_j (z,\omega)=\phi(z)\sum_{n\in\Z^2}g(2^j z-n,X_{n}(\omega)),
\end{equation}
where $g(z,\omega)$ 
decays rapidly in $z$, the random variables $X_{n}$ are i.i.d., and  $\phi\in C^\infty_0$. We establish the  
almost sure  and local uniform convergence as $j\to\infty$ of the maps
$F_j$ to a deterministic quasiconformal limit $F_\infty$.

This 
result  is   obtained as an  application of  our main theorem, 
which deals with homogenization of iterated  randomized singular integrals. As a special case of our theorem,
let $T_1,\ldots , T_{m}$ be translation and dilation invariant singular integrals on ${\bf R}^d, $ and consider a $d$-dimensional version of $\mu_j$, e.g., as defined above or within a more general setting, see Definition \ref{smf}.  We then prove that  there is a
deterministic function $f$ such that  almost surely as $j\to\infty$,
$$
\mu_j T_{m}\mu_j\ldots T_1\mu_j\to f \quad \textrm{weakly in } L^p,\quad 1 <  p < \infty\ . 
$$
\end{abstract}

\maketitle

\section{Introduction and statement of results}

\subsection{Background and motivation - a birds eye view}\label{ss:bird}
The purpose of this paper is twofold: We initiate a systematic study of random quasiconformal homeomorphisms,
and we develop a framework for homogenization of iterated singular integrals. Our main results regarding
the former topic will be obtained as consequences of our results regarding the latter,
which are of independent interest.
Since the precise statements of our results require some preparation, in this section
we give a brief and informal description of our work. 

Recall that quasiconformal maps are homeomorphic $W^{1,2}_{loc}-$solutions of the 
Beltrami equation 
\begin{equation}\label{eq:beltrami}
\partial_{\overline z} F = \mu \partial_z F,
\end{equation}
and that for any measurable function $\mu \colon \C\to \C$ with $\|\mu\|_\infty<1$
there is an essentially unique quasiconformal solution.
Recent developments have shown an emerging need for a theory of random quasiconformal maps. 
For example, simple closed planar curves can be described via their welding homeomorphism, and random loops such as those associated with the Schramm-Loewner evolution SLE lead to random circle homeomorphisms. Beginning with the work of Sheffield, these welding homeomorphisms can be described in terms of Liouville Quantum Gravity.
It is still open to analytically solve the ``welding problem'' of re-constructing the loops from these homeomorphisms. The standard approach of solving welding problems is via the Beltrami equation \eqref{eq:beltrami}, leading to random Beltrami coefficients $\mu$ in the case of random weldings. Progress towards solving this problem has been made in \cite{AJKS}.

There are also other cases in random geometry where quasiconformal mappings arise naturally. For instance,  certain scaling limits of domino tilings  \cite{CKP}, and more generally of dimer models \cite{KOS}, exhibit different limiting phases. Quasiconformal mappings appear particularly useful in describing their geometry  \cite{ADPZ}. Moreover,
there is a connection to homogenization of random conductance models, which in turn can be thought of as a special case of  Brownian motion in a random environment. Here we refer to the review \cite{B}.

In another direction, in material sciences it is important to understand  random materials structures, modelled by elliptic PDE's, and look for global or homogenised properties  of the material. From the  vast literature on homogenization of random PDE's we mention as examples \cite{PV1},\cite{GO}, \cite{AKM}, where the last mentioned monograph contains an extensive bibliography. 

\bigskip

In the present paper we will approach the  Beltrami equation \eqref{eq:beltrami}, with a random coefficient $\mu$, via the method of singular integral  operators.
We will mostly work with solutions normalized by 
\begin{equation}\label{3point}
F(w)=w\quad\text{ for }\quad w\in\{ 0,1,\infty\}.
\end{equation} 
However, in the special deterministic case where $\mu$ happens to be compactly supported, 
it is often more convenient to work with the
unique homeomorphic solution to  \eqref{eq:beltrami} that has  the \emph{hydrodynamic normalization}
\begin{equation}\label{hydrodynamic}
F(z)-z=o(1) \quad \textrm{as}\quad z\to\infty.
\end{equation}
This so-called 
\emph{principal solution} to \eqref{eq:beltrami} 
can be obtained 
from the Neumann series\footnote{Operators and multipliers in this paper are always applied from right to left unless otherwise specified, thus for instance $\mu T \mu T \mu = \mu T( \mu T\mu)$.}
$$\deeb F = \mu + \mu T\mu + \mu T\mu T \mu +  \ldots
$$
with $T\,$ a specific singular integral operator, the Beurling transform, see \eqref{eq:beurl}  below.  

Therefore we are naturally led to the study of homogenisation phenomena for iterated singular integral operators. Here it is useful to consider the problem from a broader point of view. Our main result on homogenised iterated singular integrals shows that this can be carried out in surprising generality, allowing for flexibility and a wide range of potential applications:

\begin{theorem}\label{th:main22} For each $1\leq k\leq m-1$ let 
$T_k$ be a translation and dilation invariant singular integral. Further, let $\mu^{(1)} = \mu^{(1)}_\delta,\ldots,\mu^{(m)} = \mu^{(m)}_\delta$ be stochastic multiscale functions. Then for any $p\in(1,\infty)$
the iterated singular integral
$$
h_\delta\coloneqq \mu^{(m)}_\delta T_{m-1}\mu^{(m-1)}_\delta\ldots \mu^{(2)}_\delta T_1\mu^{(1)}_\delta
$$
converges weakly in $L^p(\R^d)$ to a deterministic limit function
as $\delta\to 0$ (convergence in probability). 
For the subsequence $h_{2^{-k}}$ the weak convergence takes place almost surely.
\end{theorem}
The stochastic multiscale functions above are a large class of random functions with $\delta-$periodic statistical structure. Their precise definition  is given in Section \ref{random-sec},  and  Section \ref{sec:proof}  is devoted to the proof of Theorem \ref{th:main22}.  
In general,  the multiscale functions need not be bounded or compactly supported. An example of such function is provided by \eqref{eq:expl}.

In the next subsection we present a variety of natural and specific random Beltrami equations $\partial_{\overline z} F_\delta = \mu_\delta \partial_z F_\delta$, where the coefficients $\mu_\delta$ are stochastic multiscale functions  with $\| \mu \|_{\infty}$ bounded by some $k<1$. 
 To complete the picture  we then need methods more specifically related to quasiconformal mappings  to show that the corresponding random solutions $F_\delta$ have almost surely a unique deterministic normalised quasiconformal limit $F_\infty$, see e.g., Theorem \ref{th:main} below.

Finally, we mention that  Theorem \ref{th:main22}  also applies  to many basic homogenization problems of  random partial differential operators, see Example \ref{example:general}
 in Section \ref{ss:sing_int_homogenization}.

\subsection{Quasiconformal homogenization}\label{ss:qc_homogenization}

In this subsection we state our main results on quasiconformal homogenization and illustrate them by means of several model examples  
 of coefficients $\mu_\delta$.
We consider both \emph{deterministic} and \emph{random} quasiconformal maps, though our main emphasis is on the latter case.  It will be convenient to adopt the following rescaling notation:

\begin{definition}[Rescaling notation]\label{Rescale}  If $\delta>0$, $n \in \Z^d$, and $g \colon \R^d \to \C$ is a function, we define the rescaled function $g_{[n,\delta]} \colon \R^d \to \C$ by the formula
$$ g_{[n,\delta]}(x) \coloneqq g \left( \frac{x}{\delta} - n \right).$$
For instance, we will apply this convention to the weight
$$ \langle x \rangle \coloneqq (1 + |x|^2)^{1/2}$$
so that
$$ \langle x \rangle_{[n,\delta]} = \left( 1 + \left| \frac{x}{\delta}-n \right|^2 \right)^{1/2}$$
for any $x \in \R^d$, $\delta>0$, and $n \in \Z^d$.

More generally, if $g \colon \R^d \times \Omega \to \C$ is a function of a spatial variable $x \in \R^d$ and a supplementary variable $\omega \in \Omega$, we define $g_{[n,\delta]} \colon \R^d \times \Omega \to \C$ by the formula
$$ g_{[n,\delta]}(x,\omega) \coloneqq g \left( \frac{x}{\delta} - n, \omega \right).$$
We extend this convention to the complex plane $\C$ by identifying $\C$ with $\R^2$ (and $\Z^2$ with the Gaussian integers $\Z[i]$.
\end{definition}

We will typically apply this convention with functions $g$ that are concentrated near the unit ball $B(0,1)$, in which case the rescaled function $g_{[n,\delta]}$ will be concentrated near the ball $B(n\delta,\delta)$. Conversely, the weight $\langle \cdot \rangle_{[n,\delta]}$ is small in $B(n\delta,\delta)$ and large elsewhere.

\medskip
\noindent
{\bf Model 1:} 
 The deterministic function
\begin{equation}\label{mu-ex1}
\mu_\delta(z) \coloneqq  \varphi(z) \sum_{n\in\Z^2}a_{[n,\delta]}(z),
\end{equation}
where $\varphi \in C^\infty_0(\C)$ is a test function and $a\colon \C \to \C$ is a smooth non-constant  function supported on $[0,1]^2$, and the rescaling $a_{[n,\delta]}$ is defined by Definition \ref{Rescale}.  One assumes that $\|\varphi\|_\infty \| a\|_\infty<1.$

\medskip
\noindent
{\bf Model 2:} Here $\mu_\delta$ is a  random function  given by either
\begin{equation}\label{mu-ex2}
\begin{split}
\mu_\delta(z) &\coloneqq a 1_{Q_0}(z) \sum_{n \in \Z^2} \varepsilon_n (1_{Q_0})_{[n,\delta]}(z) \\
&=  a 1_{Q_0}(z) \sum_{n \in \Z^ 2} \varepsilon_n 1_{n\delta + [0,\delta]^2}(z),
\end{split} \end{equation}
or 
\begin{equation}\label{mu-ex3,5}
\mu_\delta(z) \coloneqq  a \sum_{n \in \Z^2} \varepsilon_n 1_{n\delta + [0,\delta]^2}(z),
\end{equation}
 where $a \in \C$ satisfies $|a|<1$, and $Q_0 \coloneqq [0,1]^2$ is the unit square with corners $0, 1, i, 1+i$, the $\varepsilon_n \in \{-1,+1\}$ are i.i.d. random signs, and $n\delta + [0,\delta]^2$ is the square of sidelength $\delta$ and bottom left corner equal to $n\delta$, $n \in \Z^2$. We could as well allow the $\varepsilon_n$ to be arbitrary i.i.d.  random variables with $|\varepsilon_n|\leq1$.

\medskip
\noindent
{\bf Model 3:} A  more general model is obtained by allowing
the independent `bumps' to have non-compact support and adding
an envelope factor that  varies the size of $\mu$ locally, and is  independent of the scaling $\delta$. Thus,
let $g$ be a rapidly decreasing function and define the random `bump field'
\begin{equation}\label{eq:bf}
B_\delta =  \sum_{n \in \Z^2} \varepsilon_n g_{[n,\delta]},
\end{equation}
where $\varepsilon_n$ are any i.i.d random variables, the $g_{[n,\delta]}$ are defined by Definition \ref{Rescale} and we assume the pointwise bound $|B_\delta|\leq 1.$
Then set
\begin{equation}\label{mu-ex3}
\mu_\delta \coloneqq  \phi 1_U  B_\delta,
\end{equation}
where the `envelope function' $\phi$ satisfies the pointwise bound $|\phi|\leq k$ for some $k<1$ and is H\"older continuous with 
some exponent $\alpha>0$, and $U\subset\C$ is a  domain with
piecewise H\"older-boundary (e.g., $U$ could as well be the whole plane).

  If we  specialize to the case $\phi\equiv a$,
  where $a$ is a complex constant with $|a|<1,$ 
  $\mu_\delta$ becomes a constant multiple of the random  bump field \eqref{eq:bf}:  
\begin{equation}\label{mu-ex2,5}
\mu_\delta(z) \coloneqq  a  B_\delta(z).
\end{equation}

\medskip

In each of the above model cases,
let $F_\delta$ be the unique solution to the (random or deterministic) Beltrami equation
\begin{equation}\label{eq:randombeltrami}
\deeb F_\delta =\mu_\delta \, \dee F_\delta
\end{equation}
with $3$-point normalization \eqref{3point}.
The  basic question of quasiconformal homogenization then asks if
 the sequence $F_{2^{-k}}$ converges as $k\to\infty$. 
We answer this question  by showing that  there is almost sure convergence to a deterministic limit homeomorphism.

We will prove a  general result that covers all the above models as special cases,  and is substantially of more general nature. In order to state the result we need to define the 
admissible 
envelope functions and  
random bump fields. 
\begin{definition}[Random  bump fields]\label{de:rbf} We define \emph{random bump data} to be a pair $(g,X)$, where $X$ is a random variable taking values in\footnote{We place our random parameter in the space $\R$ for sake of concreteness, but this space could be replaced by a more general measurable space, e.g., $\R^d$ for any $d$, if one wished.} $\R$, and $g\colon \C \times \R \to \C$ be a measurable function  with rapid  decrease in the first variable,
 \begin{equation}\label{mu-decay}
|g(z,y)|\leq C_M \langle z \rangle^{-M} \quad \textrm{for all}\;\; M\geq 1\quad
\textrm{and}\;\; z\in\C, y\in\R
\end{equation}
which obeys the pointwise bound
$$
\left|\sum_{n\in\Z^2} g(z-n,y_n)\right|\leq 1
$$
for all $z\in\C$ and all real sequences $(y_n)_{n\in\Z^2}$.  We define a \emph{random bump field} with data $(g,X)$ and scaling parameter $\delta>0$ to be a random field of the form
\begin{equation}\label{eq:rbf1}
B_\delta(z)\coloneqq \sum_{n\in\Z^2}g_{[n,\delta]}(z,X_n)
\end{equation}
where the rescaling $g_{[n,\delta]}$ is defined by Definition \ref{Rescale}, and $X_n, n \in \Z^2$ are independent copies of the random variable $X$.
\end{definition}

In turn, the admissible envelope functions are as follows:

\begin{definition}[Beltrami envelope functions]\label{de:ref} 
A measurable function $\phi\colon \C\to\C$ is a \emph{Beltrami envelope function}
if there is $k\in (0,1)$ such that $|\phi(z)|\leq k$ for almost every  $z\in \C$ and $\phi$ is locally H\"older-continuous
in $L^1$-norm: there is $\alpha>0$ such that for any $R>0$ there is 
$C_R<\infty$ with
$$
\| \Delta_h (1_{B(0,R)}\phi)\|_{L^1(\C)}
\leq C_R|h|^\alpha,\quad\textrm{for}\;\; |h|\leq 1
$$
where the difference operator $\Delta_h$ is defined by
\begin{equation}\label{difference-def}
\Delta_h f(x) \coloneqq f(x+h) - f(x).
\end{equation}
\end{definition}
 \begin{example}\label{ex:envelope} Assume that $\phi\colon \C\to\C$ is
$\alpha$-H\"older continuous and satisfies $|\phi|\leq k<1$.  Assume also that $U\subset \complex$ is a domain with locally H\"older-regular boundary. Then it is easy to verify that $1_U \phi$
is a Beltrami envelope function. This holds also true if  (locally) the Minkowski dimension of $\partial U$
is strictly less than  $2$.
\end{example}

 In each of  the models 1-3 above
the random dilatation can be written in the form 
$\mu_\delta=\phi \, B_\delta,$ where $\phi$ is a Beltrami envelope function
and $B_\delta$ a random bump field. Hence our result on quasiconformal homogenization ,
to be stated next, covers all these  cases.

\begin{theorem}\label {th:main} Let $(g,X)$ be random bump data, and let $\phi$ be a Beltrami envelope function. For $\delta >0$ let
$$,
\mu_\delta(z) = \phi (z) B_\delta (z),
$$ 
where $B_\delta$ is the random bump field 
\eqref{eq:rbf1}
determined by $g,X$.
Denote by $F_j$, $j\geq 1$, the $3$-point normalized solution to the random Beltrami equation
\begin{equation}\label{eq:Fj}
\deeb F_j =\mu_{2^{-j}} \dee F_j\,  .
\end{equation}
\begin{itemize}
\item[(i)] There is a unique deterministic limit function $F_\infty$ such that $F_\infty\colon \C\to \C$ is a quasiconformal homeomorphism  and as $j\to\infty$, almost surely
$$
F_{j}\to F_\infty\quad \textrm{locally uniformly.}
$$
\item[(ii)]   Assume that the envelope function
 $\phi$ is continuous at $z_0$. Then  the dilatation $\mu_{F_\infty}$ of the limit function $F_\infty$ is continuous at $z_0$, 
 and $\mu_{F_\infty}(z_0)$ 
 depends only on the random bump data $(g,X)$ and on the value $\phi(z_0)$.  More precisely, one has  
 $$
 \mu_{F_\infty}(z_0)= h_{(g,X)}( \phi(z_0)),
 $$
 where the function $h_{(g,X)}\colon \{ |z|<1\}\to \{ |z|<1\}$  is continuous.
 \item[(iii)] If the random variables $\varepsilon_n$ are symmetric, the limit $F_\infty$ in  both cases of  Model 2,   \eqref{mu-ex2} and \eqref{mu-ex3,5},
  is given by the identity map, $F_\infty(z)=z$ for all $z$.  This is not 
 necessarily the case in the more general setting of \eqref{mu-ex2,5}.
\end{itemize}
 \end{theorem}
 
 \medskip
 
\noindent The    proof 
of this theorem is contained in Section \ref{se:qchomogenization},
which also contains other related results and remarks. In particular, the above theorem 
applies as well to the deterministic homogenization problem.
We also stress that  the coefficient $\mu$ need not be compactly supported, in spite of the fact that the proof is based on the Neumann series.

\begin{figure}[t]
\centering
  \includegraphics[height=125mm]{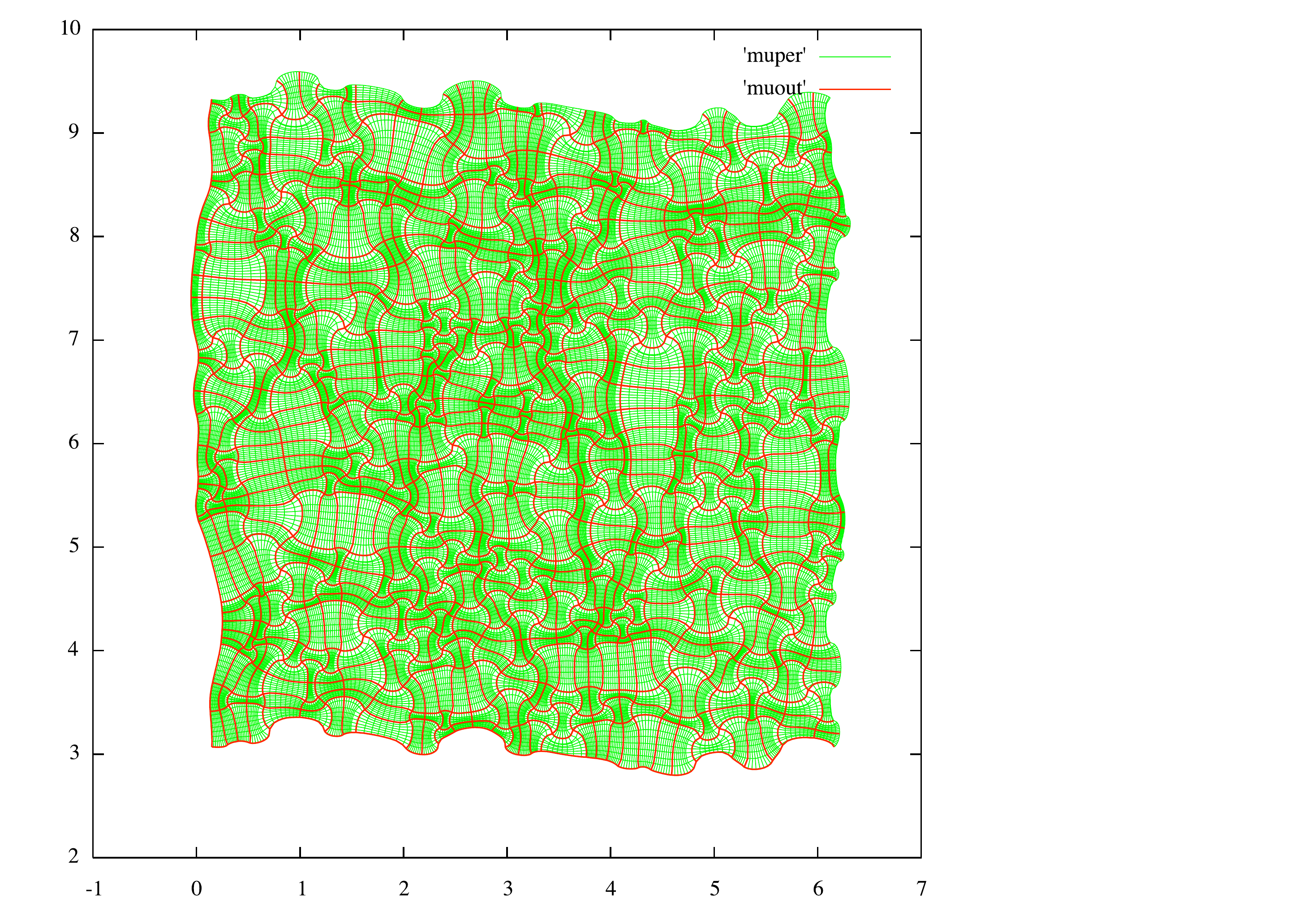}
  \caption{A random qc-map obtained by Model 2 with
  $a=1/2.$
  We thank David White for help in producing the picture.}
  \label{whitfig}
 \end{figure}

\begin{Remark}\label{re:explain} One should note that in the above
result there is no need for the stochastic bump fields corresponding to
different $\delta$'s to be independent. Indeed, their stochastic relation can be arbitrary. This can be understood by by writing the
dilatation of $F_j$ in the form
$$
\mu_{F_j}(z)=\phi(z)\left(\sum_{n\in \Z^2}g_{[n,2^{-j}]}(z,X_{n,j})\right),
$$
where $X_{n,j}\sim X,$ for each $n,j$, and only for each fixed $j$ the
random variables $X_{n,j}$, $n\in\Z^2$, are assumed to be independent. Thus
there can be arbitrary stochastic relations between the different
layers $(X_{n,j})_{n\in\Z^2}$ and  $(X_{n,j'})_{n\in\Z^2}$
for $j\not=j'.$ In particular, this possible dependence structure between different scales does not affect the deterministic limit function $F_\infty$,
which depends only on the triplet $(\phi,g,X).$  The main reason for this is that the failure probability in our main estimate (Theorem \ref{mainthm}) decays polynomially in $\delta$ (and hence exponentially in $j$ if $\delta = 2^{-j}$).

Let us also point out that for the sake of simplicity we leave out many considerations that would be  possible via the techniques of the present paper. For example, one may relax the speed of convergence to zero    in the  subsequence $B_{2^{-j}}$, and it is possible to consider quasiconformal maps between arbitrary domains. 
\end{Remark}

We also present a homogenization result for mappings of finite distortion, i.e.
we consider random dilatations $\mu_\delta$ with $\|\mu_\delta\|_{L^\infty(\R^2)}=1.$
An example of this kind of dilatation is given by
\begin{description}
\item[Model 4] A random function as in the model example \eqref{mu-ex2}
\begin{equation}\label{mu-ex4}
\begin{split}
\mu_\delta(z) &\coloneqq 1_{Q_0}(z) \sum_{n \in \Z^2} \varepsilon_n (1_{[0,1]^2})_{[n,\delta]}(z) \\
&=   1_{Q_0}(z) \sum_{n \in \Z^2} \varepsilon_n 1_{n\delta + [0,\delta]^2}(z),
\end{split}
\end{equation}
\end{description}
but now with random i.i.d. variables $\eps_n$ such that $|\eps_n| < 1$ and $(1-|\eps_n|)^{-1}$
has sufficiently fast exponential decay.
Theorems \ref{degeco:2.1}, \ref{th:last_deterministic} and \ref{th:last_random} in Section \ref{se:qchomogenization} generalize Theorem \ref{th:main}
to the degenerate case \eqref{mu-ex4} and beyond.
\medskip

It is tempting to try to prove almost sure convergence of $F_\delta$  in the above examples solely using  weak convergence of  $\mu_\delta$. However,
it is important to note that this is impossible, as the following example illustrates.
 Some deeper properties of $\mu_\delta$ and their interaction with singular integrals are involved here.

\begin{example}\label{ex1} Let $a\in (-1,1)$ and define the Beltrami coefficient  $\nu(z)$ that is $2$-periodic in the $x$-variable and constant in the $y$-variable by setting
$$
\nu (z)\coloneqq \begin{cases} a& \textrm{if}\; x\in [2n,2n+1), \quad n\in\Z,\\
-a& \textrm{if}\; x\in [2n+1,2n+2), \quad n\in\Z .
\end{cases}
$$
Write $b=\frac{1-a}{1+a}$ and observe that the function
$$
g(x+iy)\coloneqq \begin{cases} (x-2n)+n(1+b^2)+iby& \textrm{if}\; x\in [2n,2n+1), \quad n\in\Z,\\
b^2(x-(2n+1))+n(1+b^2)+1+iby& \textrm{if}\; x\in [2n+1,2n+2), \quad n\in\Z 
\end{cases}
$$
solves $g_{\overline{z}}=\nu g_z.$
Now consider the homogenized dilatation $\mu_j(z)\coloneqq \nu(2jz)$ for any integer $j\geq 1$ and let 
$F_j$ satisfy  $$\deeb F_j=\mu_j\, \dee F_j$$
with the three-point normalization, so that
$$
F_j(z)= \frac{g(2jz)}{j(1+b^2)}.
$$
As $j\to\infty,$ it is clear that $\mu_j$ converges locally weakly to
zero. However, by the above formulas we see that there is the uniform convergence $F_j\to F_\infty$, where 
$$
F_\infty(x+iy)=x+\frac{1-a^2}{1+a^2}iy,
$$
and $\mu_{F_\infty}\equiv a^2$ identically. By  considering the sequence $\widetilde \mu_j$, where $\widetilde \mu_{2j}=\mu_j$ and $\widetilde \mu_{2j+1}=0$ 
we obtain a locally weakly null sequence of  dilatations for which the homogenization limit does not exist. Finally, we observe that its is easy
to localize this observation and obtain the same phenomenon for compactly supported dilatations (see Lemma \ref{le:locality} below).
\end{example}

\begin{Remark}\label{rem:oleg}  In a  recent  interesting work  \cite{IM},  Ivrii and Markovic provide a more  elementary geometric proof of some special cases of our results on quasiconformal homogenization, also allowing for non-uniform ellipticity, and give an  application to random Delaunay triangulations.

\end{Remark}

\subsection{Further remarks on Theorem \ref{th:main22}}\label{ss:sing_int_homogenization}
   
As explained in Section \ref{ss:bird}, the application of  Theorem \ref{th:main22} to quasiconformal homogenization  and solutions to  Beltrami equation \eqref{eq:beltrami} comes via the {\it Beurling transform} 
\begin{equation}\label{eq:beurl} Tg(z) \coloneqq  -\frac{1}{\pi} \operatorname{p.v.} \int_\C \frac{g(w)}{(z-w)^2}\ dw.
\end{equation}
Namely, since  $T \circ \partial_{\overline z} = \partial_z$  on $W^{1,2}(\C)$, 
finding a solution to $\partial_{\overline z} f_\delta = \mu_\delta \partial_z f_\delta$, 
with the hydrodynamic normalization $f_\delta(z) - z = o(1)$ at $z \to \infty$, is equivalent to solving the integral equation 
$ (1 - \mu_\delta T) \deeb f_\delta = \mu_\delta$. One then finds the solution  via the $L^2$-Neumann series representation
$$
\deeb f_\delta = \mu_\delta + \mu_\delta T\mu_\delta + \mu_\delta T\mu_\delta T \mu_\delta +  \ldots,
$$
and the theorem allows us to deduce the weak convergence of each single summand in the above formula. 

The Beurling transform extends  to all of  $L^2(\C)$ as an isometric isomorphism.  Moreover, it commutes with dilatations and translations.  The class of singular integrals in $\R^d$ allowed in Theorem \ref{th:main22} shares these two basic symmetries:

\begin{definition}[Singular integral operator]\label{sio-def} A dilation and  translation invariant \emph{singular integral operator} is any bounded linear operator $T\colon L^2(\R^d) \to L^2(\R^d)$ of the form
$$ Tf(x) = p.v. \int_{\R^d} \frac{\Omega( (x-y)/|x-y| )}{|x-y|^d} f(y)\ dy, \qquad \mbox{for } \; f \in C^\infty_0(\R^d),$$
where $\Omega: S^{d-1} \to \C$ is smooth and has mean zero.
\end{definition}

The definition of the general class of random multipliers considered in Theorem \ref{th:main22}, \emph{the stochastic multifunctions}, is  slightly opaque as  it  employs the notion of \emph{stochastic tensor products}. Both these notions will be explained in Section  \ref{random-sec} below. However, to give a perhaps more intuitive idea of these notions, we describe here  in  detail a special class of multifunctions which fits well  with the notions of bump-fields and Beltrami envelope functions discussed in the previous Section \ref{ss:qc_homogenization}. Thus, working in arbitrary dimension $d\geq 1$, 
fix $m\geq 1$ and  consider for each index  $1\leq \ell \leq m$  
the random function 
\begin{equation}\label{eq:random multiplier}
\mu^{(\ell)}_\delta(x)=  \phi_\ell(x) \sum_{n \in \Z^d}
(g_\ell)_{[n,\delta]}(x,X_n),
\end{equation}
where we assume for each fixed $\ell$   that:
\begin{itemize}
\item\label{envelope} The `\emph{envelope function}'  $\phi_\ell$ does not depend on $\delta$. Moreover, $\phi_\ell\in L^p(\R^d)$ for every $p\in(1,\infty)$ with the H\"older bound
$$
\| \Delta_h \phi_\ell\|_{L^p(\R^d)}
\leq C(p)|h|^{\alpha_p},\quad\textrm{for}\;\; |h|\leq 1,\quad  \textrm{where}\;\; \alpha_p>0.
$$
\item The random variables $\{ X_{n}\}_{n\in\Z^d}$ are independent and identically distributed, $X_n\sim X$ for all $n\in\Z^d.$

\item The `\emph{bump function}' $g_\ell(\cdot , \cdot)$ satisfies an $d$-dimensional analogue of condition \eqref{mu-decay}.
\end{itemize}
Lemma \ref{prod3} below implies that such  $\mu^{(\ell)}_\delta$  is a stochastic multifunction, covered by Theorem \ref{th:main22}.

 As a  last  aspect,  Theorem \ref{th:main22} applies easily to  homogenization of many random differential operators:

\begin{example}\label{example:general} 
 For  each $\ell=1,\ldots, L$, let  $P_\ell(D)$ be a constant coefficient  second order differential operator  on $\R^d$. Also let $\mu_\delta^{(\ell)}$ be random multipliers as in Theorem \ref{th:main22}. For simplicity we assume that $d\geq 3$,  and that the $\mu_\delta^{(\ell)}$ are supported on a fixed ball $B\subset\R^d$.  Our  basic ellipticity assumption  is that they satisfy a.s.
\begin{equation}\label{eq:diff0}
\sum_{\ell=1}^La_\ell\|\mu_\delta^{(\ell)}\|_{L^\infty (B)}\leq k<1  \quad \textrm{for all}\;\;\delta \in (0,1),
\end{equation}
where the constants $a_j$ will be soon defined.
We consider the following PDE on $\R^d$ with random coefficient functions
\begin{equation}\label{eq:diff}
\Delta u_\delta + \sum_{\ell=1}^L\mu_\delta^{(\ell)}P_\ell(D) u_\delta =h.
\end{equation}
Here the right hand side $h\in L^2(\R^d)$ is fixed, and is also supported in the ball $B$. We normalize the solutions $u_j$  of \eqref{eq:diff} by demanding that  $u_j(x)\to 0$ as $x\to\infty$. 

We claim that this problem has a unique solution $u_\delta\in \dot W^{2,2}(\R^d)$ converging strongly (in probability)  in $\dot W^{s,2}(\R^d)$ for every $s<2$  towards a deterministic function $u_\infty\in \dot W^{2,2}(\R^d)$ as  $\delta\to 0.$ Thus  the present homogenization problem is solvable with a deterministic limit. 

In order to sketch the argument, let us  denote by $T_\ell$ the homogeneous Fourier multiplier $T_\ell\coloneqq  \Delta^{-1}P_\ell(D)$ and note that  it is a scaling and translation invariant singular integral\footnote{\label{fn:1} Strictly speaking the $T_\ell$ might not be precisely of the form in Definition \ref{sio-def} because there may be an identity component in addition to a principal value integral; however it is a routine matter to extend the analysis in this paper to this more general setting.} on $\R^d$. We choose $a_\ell\coloneqq \sup_{|\xi|=1}|P(\xi)|$  in condition \eqref{eq:diff0}, i.e., $a_\ell$ is the $L^2$-norm of the operator $T_\ell$.  Then $f_j\coloneqq \Delta u\in L^2(B)$  satisfies the equation
\begin{equation*}\label{eq:diff2}
 f_j + \sum_{\ell=1}^L\mu^{(\ell)}Tf_j =h,
\end{equation*}
which can be uniquely solved in $L^2(\R^d)$ by the Neumann series  and, via condition \eqref{eq:diff0},  we obtain an $L^2$-convergent series
\begin{equation}\label{eq:diff3}
 f_j=  h+\sum_{\substack{1\leq \ell_1,\ldots ,\ell_m\leq L\\m\geq 1}}(-1)^m\mu^{(\ell_1)}T_{\ell_1}\mu^{(\ell_2)}T_{\ell_2}\cdots \mu^{(\ell_m)}T_{\ell_m}h.
\end{equation}
By applying the fundamental solution of the Laplacian, we see that $u_j=c_d|\cdot|^{2-d}*f_j$ solves \eqref{eq:diff} with the right behaviour in the infinity,  as $d\geq 3$. Since any other solution has the same Laplacian, they must differ by a harmonic function that vanishes at infinity, and hence their difference is zero.

Theorem \ref{th:main22} applies to each term in the sum \eqref{eq:diff3}, and together with  the uniform convergence  in $\delta$ of the series in $L^2$ we deduce  that
$f_\delta \to f_\infty$ weakly in $L^2(\R^d)$, where $f_\infty$ is also supported in $B$. The rest of the claim follows from the standard properties of the fundamental solution $c_d|\cdot|^{2-d}$.

We finally note that above the operators $P_j$ may well have lower order terms since those  produce compact Fourier multipliers between functions on fixed compact subsets of $\R^d$. Hence the terms in the Neumann-series containing them can be taken care of by multiple application of  Theorem \ref{th:main22}. Actually, we could instead of differential operators $P_j$ with constant coefficients consider as well classical pseudodifferential operators of order 2 whose principal part is a homogeneous Fourier multiplier. 

 Similarly, the technique applies to
fractional Laplacians,  and in many other type of homogenization problems. In order to spell out one more specific example -- completely without details -- consider the homogenization of the general  conductivity equation in the plane.
$$
\nabla \cdot \big( A(x)\nabla u(x)\big)=0,
$$
where the $2\times 2$ matrix $A(x)=(\delta_{j,k}+\mu_{jk}(x))_{j,k=1,2}$ is measurable and uniformly elliptic, and each $\mu_{j,k}$ is a stochastic multifunction. One may reduce this to the study of the generalized Beltrami equation $$
\dee f+\eta_1\deeb f+\eta_2\overline{\deeb f},
$$ where the coefficients $\eta_j$'s are expressed in terms of  in the matrix coefficients $\mu_{j,k}$, see e.g. \cite[Theorem 16.1.6]{AIM}. The structure of the $\eta_j$'s allows them be approximated in a suitable sense by multifunctions (see footnote \ref{fn:1} in this connection). The generalized Beltami-equation   may be solved by a  $2\times 2$-matrix valued  Neumann series, and the analysis can be then carried out analogously to  the case of the classical Beltrami equation. 

For  classical treatments of homogenization of the  above PDE's (however, not including the case of more general case of  Fourier multipliers we allow for), we refer to \cite{PV1}, \cite{PV2}, and \cite{AS}.
 
\end{example}

\subsection{Structure of the paper}
Section \ref{dmd-sec} develops the  homogenization of deterministic iterated singular integrals.
This is much easier than the random setting, but has its own  interest, and it will provide a handy tool in treating the stochastic case later on.  The admissible class of deterministic multipliers will be called  called \emph{multiscale functions} (see Definition \ref{ms}). They are defined  using the notion of  `multiscale tensor product'  (Definition \ref{mtp}), which generalizes  the product of an envelope function and a bump field. Our deterministic homogenization result is stated as Corollary \ref{singular_equivalence}.

Section \ref{random-sec} first defines the probabilistic analogues of the deterministic notions, especially the `stochastic multiscale tensor product' (Definition \ref{smtp}) is used to define \emph{stochastic multiscale functions} (Definition \ref{sms}), which are quite a bit more general than the multipliers  we discussed in Section \ref{ss:qc_homogenization}.
The general form of our main result on homogenization of randomized iterated singular integrals is formulated in Theorem \ref{mainthm} and Corollary \ref{co:main}. Lemma \ref{prod3} then verifies that the random multipliers \eqref{eq:random multiplier}  are particular instances of a  stochastic multiscale tensor product. 

The proof of Theorem \ref{th:main22} is carried out in Section \ref{sec:proof}, where it is obtained as a consequence of Theorem \ref{mainthm} and Corollary \ref{co:main}.
  Somewhat surprisingly, a considerable effort needs to be spent in establishing the convergence of the expectation of the iterated randomized integrals.  

Finally, Section \ref{se:qchomogenization} applies Theorem \ref{th:main22} to quasiconformal homogenization. There we   combine Theorem \ref{th:main22} with methods from the theory of planar quasiconformal mappings in order to show that the corresponding random solutions $F_\delta$  almost surely have a unique normalised deterministic quasiconformal limit $F_\infty$, see e.g., Theorem \ref{th:main}.

\subsection{Acknowledgements}

The authors thank Dima Shlyakhtenko for helpful discussion, and in particular for conveying the intuition (from free probability) that the contribution from crossing partitions is negligible.

The first author was supported was supported by Academy of Finland projects
1134757, 1307331 and 13316965, and ERC project 834728.
The second author was supported by NSF grants DMS-1068105 and DMS-1700069.
The  third author was supported by the Finnish Academy, grants 75166001, 1309940, and the Finnish Academy COE 'Analysis and Dynamics'.
The fourth author was supported by a grant from the MacArthur Foundation, NSF grants DMS-0649473, DMS-1266164, DMS-1764034, and a Simons Investigator Award.

\section{Deterministic multiscale functions}\label{dmd-sec}

In the sequel we  use extensively the notations $X \lesssim Y$ or $X = O(Y)$ to denote the estimate $|X| \leq CY$ where $C$ is an absolute constant.  If we need the constant $C$ to depend on some parameters, we shall indicate this by subscripts, or else indicate the dependence in the text. For instance,  $X \lesssim_p Y$ or $X = O_p(Y)$ means that $|X| \leq C_p Y$ for some constant $C$ depending on $p$. 

Our arguments in the following three sections are not specific to the Beurling transform in the plane, and so we shall work in the more general context of singular integral operators in a Euclidean space.  Accordingly, we fix a dimension $d \geq 1$; in the application to the Beltrami equation, we will have $d=2$.  We shall work with the standard Euclidean space $\R^d$, the standard lattice $\Z^d$, and the standard torus $\T^d = \R^d/\Z^d$.  We also have a scale parameter $0 < \delta < 1$, which we shall think of as being small; several of our functions shall depend on this parameter, and we shall indicate this by including $\delta$ as a subscript.

Before we can state the main result, it will be convenient to introduce a certain calculus regarding various classes of functions (namely, envelope functions, localized functions, negligible functions, and multiscale functions; we will define these classes later in this section).  To set up this calculus we shall need a certain amount of notation and basic theory.

\begin{definition}[H\"older space]  If $1 < p < \infty$ and $\alpha \in (0,1)$, we let $\Lambda^{\alpha,p}(\R^d)$ denote the space of functions $f$ whose norm
$$ \| f \|_{\Lambda^{\alpha,p}(\R^d)} \coloneqq  \|f\|_{L^p(\R^d)} + \sup_{0 < |h| < 1} \frac{\|\Delta_h f\|_{L^p(\R^d)}}{|h|^\alpha}$$
is finite, where $\Delta_h$ was defined in \eqref{difference-def}.
\end{definition}

\begin{remark}  One could also use Sobolev spaces $W^{\alpha,p}(\R^d)$ instead of H\"older spaces $\Lambda^{\alpha,p}(\R^d)$ in what follows, but we have elected to use H\"older spaces as they are slightly more elementary. Also note that
we usually use the symbol $\phi$ for a Beltrami envelope function, c.f. Definition \ref{de:ref}. 
\end{remark}

We recall the following definition from p. \pageref{envelope}.

\begin{definition}[Envelope function]\label{def:env} An \emph{envelope function} is a function $f: \R^d \to \C$ (not depending on the scale parameter $\delta$) such that for every $1 < p < \infty$ there exists $\alpha > 0$ such that $f \in \Lambda^{\alpha, p}(\R^d)$.  Thus the space of all envelope functions is
$$ \bigcap_{1 < p < \infty} \bigcup_{\alpha \in(0,1)} \Lambda^{\alpha, p}(\R^d).$$
\end{definition}

\begin{example} If $Q$ is a cube in $\R^d$, then one checks that
$1_Q\in \Lambda^{\alpha, p}(\R^d)$ for $\alpha\in (0,1/p)$ so that
the indicator function $1_Q$ is an envelope function.  Any function in the Schwartz class is an envelope function.
\end{example}

\begin{lemma}\label{envprod}  The product of two envelope functions is again an envelope function.
\end{lemma}

\begin{proof}  From H\"older's inequality one quickly sees that the product of two functions in $\Lambda^{\alpha, p}(\R^d)$ lies in $\Lambda^{\alpha, p/2}(\R^d)$.  The claim follows.
\end{proof}

\begin{definition}[Localized function]  A (deterministic) \emph{localized function} is a function $g: \R^d \to \C$ (not depending on the scale parameter $\delta$) such that for every $1 < p < \infty$ and $N > 0$, the function $\langle \cdot \rangle^N g$ lies in $L^p(\R^d)$, where $\langle \cdot \rangle$ is as in  Definition \ref{Rescale}.  Thus the space of all localized functions is
$$ 
\bigcap_{1 < p < \infty} \bigcap_{N > 0}\langle \cdot \rangle^{-N}L^p(\R^d),
$$
\end{definition}

\begin{example} The indicator function $1_Q$ of a cube is a localized function, as is any function in the Schwartz class.
\end{example}

We shall often exploit the ability of localized functions to absorb arbitrary powers of $\langle \cdot \rangle_{[n,\delta]}$ via the following lemma, which improves upon the triangle inequality in $L^p$ at the cost of inserting different localizing weights $\langle \cdot \rangle_{[n,\delta]}$ on each summand.

\begin{lemma}[Localization lemma]\label{loc}  Let $\delta > 0$ and  let $1 < p < \infty$.  Then we have the estimate
$$ \| \sum_{n \in\Z^d} f_n \|_{L^p(\R^d)} \lesssim_{p,d} (\sum_{n \in \Z^d} \| \langle \cdot \rangle_{[n,\delta]}^d f_n \|_{L^p(\R^d)}^p)^{1/p}$$
for any sequence $f_n \in L^p(\R^d)$ of functions.
\end{lemma}

\begin{proof}  We can rescale $\delta=1$.  It will suffice to prove the pointwise inequality
$$ |\sum_{n \in \Z^d} f_n| \lesssim_{p,d} (\sum_{n \in \Z^d} \langle \cdot \rangle_{[n,1]}^{pd} |f_n|^p )^{1/p}.$$
By H\"older's inequality,  it is enough to  show that  pointwise
$$ (\sum_{n \in \Z^d} \langle \cdot \rangle_{[n,1]}^{-p' d})^{1/p'} \lesssim_{p,d} 1$$
where $p' = p/(p-1)$ is the dual exponent of $p$.  But this can be established by direct calculation.
\end{proof}

Let us record couple of   elementary  properties of localized functions.

\begin{lemma}\label{localized} \text{}
\begin{itemize}
\item[(i)] The product of two localized functions is a localized function.
\item[(ii)] For any localized function $g$ and any sequence $(a_n)$ it holds that 
$$
\|\sum_{n\in\Z^d}a_n g_{[n,\delta]}\|_{L^p(\R^d)}\lesssim_{p,g}\delta^{d/p} \|(a_n)\|_{\ell^p},\quad 1<p<\infty
$$
where $g_{[n,\delta]}$ is given by Definition \ref{Rescale}.
\end{itemize}
\end{lemma}

\begin{proof} Claim (i) follows from H\"older's inequality, and claim (ii) is an immediate consequence of Lemma \ref{loc}.
\end{proof}

\begin{definition}[Discretization]  Let $f$ be an envelope function, and let $0 < \delta < 1$.  We define the \emph{discretization} $[f]_\delta\colon \Z^d \to \C$ of $f$ at scale $\delta$ to be the function
$$[f]_\delta(n) \coloneqq  \frac{1}{\delta^d} \int_{n\delta + [0,\delta]^d} f,$$
thus $[f]_\delta(n)$  is the average value of $f$ on the cube $n\delta + [0,\delta]^d$.
\end{definition}

\begin{definition}[Multiscale tensor product]\label{mtp}  Let $f$ be an envelope function and $g$ be a localized function.  We define the \emph{multiscale tensor product} $f \otimes_\delta g \colon \R^d \to \C$ of $f$ and $g$ to be the function
$$ f \otimes_\delta g \coloneqq  \sum_{n \in \Z^d} [f]_\delta(n) g_{[n,\delta]}$$
where $g_{[n,\delta]}$ is defined by Definition \ref{Rescale}.
\end{definition}

\begin{definition}[Negligible function]  A function $F = F_\delta \colon \R^d \to \C$ depending on the parameter $0 < \delta < 1$ is said to be \emph{negligible} if for every $1 < p < \infty$ there exists $\eps_p > 0$ and $C_p > 0$ such that 
$$ \|F_\delta \|_{L^p(\R^d)} \leq C_p \delta^{\eps_p}$$
for all $0 < \delta < 1$.
\end{definition}

\begin{remark}  Note in particular that if $F$ is negligible, then $F_\delta$ converges to zero in $L^p$ norm as $\delta \to 0$ for every $1 < p < \infty$, and furthermore the same is true even if one multiplies $F_\delta$ by an arbitrary power of $(\log \frac{1}{\delta})$.  This freedom to absorb logarithmic factors in $\delta$ will be useful for technical reasons later in this paper.
\end{remark}

\begin{definition}[Multiscale function]\label{ms}  A function $F = F_\delta \colon \R^d \to \C$ depending on the parameter $0 < \delta < 1$ is said to be a (deterministic) \emph{multiscale function} if it has an expansion
$$ F_\delta = \sum_{j=1}^J f_j \otimes_\delta g_j + G_\delta$$
where $J \geq 1$ is an integer, $f_1,\ldots,f_J$ are envelope functions, $g_1,\ldots,g_J$ are localized functions, and $G_\delta$ is a negligible function.  If  $F_\delta$ and $\widetilde F_\delta$ are two multiscale functions such that the difference $F_\delta-\widetilde F_\delta$
is negligible, we say that $F_\delta$ and $\widetilde F_\delta$
are \emph{equivalent}.
\end{definition}

\begin{example}  If $Q$ is a cube, and $g$ is a localized function, then the function
$$ F_\delta(x) \coloneqq  \sum_{n \in \Z^d} 1_Q(n \delta) g_{[n,\delta]}(x)$$
can be easily verified to be a multiscale function. 
To this end we use Lemma \ref{localized}(ii) to estimate
\begin{align*}
\|F_\delta - 1_Q \otimes_\delta g \|_{L^p(\R^d)}\;\;
&\leq \;\;\left\|\sum_{n\textrm{ : } d(n\delta,\partial Q)\leq 2\sqrt{d}\delta} |g_{[n,\delta]}|\right\|_{L^p(\R^d)} 
\lesssim \big(\delta^{(1-d)}\big)^{1/p}\delta^{d/p}=\delta^{1/p} .
\end{align*}
Hence the difference $F_\delta - 1_Q \otimes_\delta g$ is  negligible.   A similar statement is true  if $1_Q$ is replaced by a Schwartz function. 
\end{example}

\begin{example} The function $\mu_\delta$ defined in \eqref{mu-ex1} is a multiscale function.  Indeed, $\mu_\delta$ is equivalent to $\varphi \otimes_\delta a .$
More generally, we will prove in Lemma \ref{prod3} below that if  $g$  is bounded and quickly decaying,
$$
|g(x)|\leq C_N \langle x \rangle^{-N}\quad \textrm{for all}\;\; N\geq 1\;\;\textrm{and}\;\; x\in\R^d,
$$
then for any envelope function $f$ the stochastic multiscale tensor product $f\otimes_\delta g$ is equivalent to the function $\displaystyle f\sum_{n\in \Z^d}g_{[n,\delta]}.$
\end{example}

We continue with basic  discretization estimates for envelope functions, encoded in the following two lemmas.

\begin{lemma}\label{discret} Let $f$ be an envelope function.  
\begin{itemize}
\item[(i)] One has
 $\|[f]_\delta \|_{\ell^p(\Z^d)} \lesssim_{p,d} \|f\|_{L^p(\R^d)}\delta^{-d/p}$
for all $0 < \delta < 1$ and $1 < p < \infty$. 
\item[(ii)] For any $r \in \Z^d$ we have
$$ \|\Delta_r [f]_\delta \|_{\ell^p(\Z^d)} \lesssim_{p,f,d} (r \delta)^{\eps_p} \delta^{-d/p}$$
for some $\eps_p > 0$ independent of $\delta$ or $r$.
\end{itemize}
\end{lemma}

\begin{proof}  From H\"older's inequality followed by Fubini's theorem we have
$$ \| [f]_\delta \|_{\ell^p(\Z^d)}\leq
\left\| \left(\frac{1}{\delta^d} \int_{n\delta+[0,\delta]^d} |f|^p\right)_{n\in \Z^d}^{1/p} \right\|_{\ell^p(\Z^d)} = \delta^{-d/p} \|f\|_{L^p(\R^d)}$$
and (i) follows since $f \in L^p(\R^d)$.  
For (ii), observe that 
$$
\Delta_r [f]_\delta =[\Delta_{\delta r} f]_\delta.
$$
The claim now follows from (i) as $f \in \Lambda^{\eps_p,p}(\R^d)$ for some $\eps_p > 0$.
\end{proof}

Also, discretization and multiplication almost commute:

\begin{lemma}\label{commutator}  Let $f, F$ be envelope functions.  Then for every $1 < p < \infty$ there exists $\eps_p > 0$ such that $\| [fF]_\delta - [f]_\delta [F]_\delta \|_{\ell^p(\Z^d)} \lesssim_{p,f,F,d} \delta^{-d/p + \eps_p}$.
\end{lemma}

\begin{proof}  From Fubini's theorem we have
$$ [fF]_\delta(n) - [f]_\delta(n) [F]_\delta(n) = \frac{1}{\delta^{2d}} \int_{\delta n+[0,\delta]^d} \int_{\delta n+[0,\delta]^d}
f(x)(F(x) - F(y))\ dx dy.$$
Writing $y = x + r$ we can thus estimate
$$ |[fF]_\delta(n) - [f]_\delta(n) [F]_\delta(n)| \leq \frac{1}{\delta^{2d}} \int_{[-\delta,\delta]^d} \int_{\delta n+[0,\delta]^d}| f(x) \Delta_r F(x)|\ dx dr,$$
and hence by Minkowski's inequality
$$\| [fF]_\delta - [f]_\delta [F]_\delta \|_{\ell^p(\Z^d)} \leq
\frac{1}{\delta^{d}} \int_{[-\delta,\delta]^d}
\left\| \frac{1}{\delta^{d}} \int_{\delta n+[0,\delta]^d} |f(x) \Delta_r F(x)|\ dx \right\|_{\ell^p_n(\Z^d)}\ dr.$$
Lemma \ref{discret}(i) allows us to  conclude
$$\| [fF]_\delta - [f]_\delta [F]_\delta \|_{\ell^p(\Z^d)} \leq
\frac{\delta^{-d/p}}{\delta^{d}} \int_{[-\delta,\delta]^d}
\| f \Delta_r F \|_{L^p(\R^d)}\ dr,$$
and finally by H\"older's inequality and the fact that $f \in L^{2p}(\R^d)$ and $F \in \Lambda^{\eps_{2p}, 2p}(\R^d)$ for some $\eps_{2p} > 0$ we see that for $r\in [0,\delta]^d$
$$ \| f \Delta_r F\|_{L^p(\R^d)} \lesssim_{p,f,F} \delta^{\eps_{2p}}$$
and the claim follows.
\end{proof}

Next we consider the basic properties  of multiscale functions. 
For this purpose we need a couple of useful lemmas. 

\begin{lemma}\label{Lp-multiscale} Assume that $f$ is an envelope function and $g$  a localized function. Then for any $\delta >0$
\begin{equation}\label{Lp-multiestimate}
\| f\otimes_\delta g \|_{L^p(\R^d)}\lesssim_{p,d}\|f\|_{L^p(\R^d)}
\|\langle \cdot \rangle^d g\|_{L^p(\R^d)}.
\end{equation}
\end{lemma}
\begin{proof}
We apply
the localization lemma (Lemma \ref{loc}) to estimate
\begin{align*}
\| f\otimes_\delta g \|_{L^p(\R^d)}&=
\|\sum_{n \in \Z^d} [f]_\delta(n) g_{[n,\delta]}\|_{L^p(\R^d)}\\
&\lesssim_{p} \left(\sum_{n\in\Z^d} [f]^p_\delta(n)\| \langle \cdot \rangle_{[n,\delta]}^d g_{[n,\delta]} \|^p_{L^p(\R^d)}\right)^{1/p}\\
&\lesssim_{p,d} \delta^{-d/p}\|f\|_{L^p(\R^d)}
\delta^{d/p}\|\langle \cdot \rangle^dg\|_{L^p(\R^d)}
\end{align*}
where we applied Lemma \ref{discret}(i) and the observation
  $\| \langle \cdot \rangle_{[n,\delta]}g_{[n,\delta]} \|^p_{L^p(\R^d)}=
\delta^{d}\| \langle \cdot \rangle g \|^p_{L^p(\R^d)}$ for all $n.$ 

\end{proof}
\noindent In particular, if supp$(g)\subset[0,1]^d,$ then the above lemma yields the simple estimate
$$
\| f\otimes_\delta g \|_{L^p(\R^d)}\lesssim_{p,d}\|f\|_{L^p(\R^d)}
\|g\|_{L^p(\R^d)}.
$$

The following lemma reduces us to considering multiscale tensor products $f\otimes_\delta g$ with $g$ supported in $[0,1]^d.$

\begin{lemma}\label{adecay} Assume that $f$ is an envelope function and $g$
is either a  localized function, or $($more generally$)$ that it satisfies for each $p\in (0,\infty)$
\begin{equation}\label{decay}
\|g1_{k+[0,1]^d}\|_{L^p(\R^d)}\lesssim_{g,p}\langle k \rangle^{-a},\quad k\in\Z^d,
\end{equation}
with some $a>d$. 
Then $f\otimes_\delta g$ is a multiscale function that is equivalent to $f\otimes_\delta \widetilde g$, where $\widetilde g$ is supported in $[0,1]^d$ and given explicitly by the formula
\begin{equation}\label{discret1}
\widetilde g(x)\coloneqq 1_{[0,1]^d}(x)\sum_{k\in\Z^d}g(x+k).
\end{equation}
\end{lemma}
\begin{proof}
Observe first that 
any localized function satisfies \eqref{decay}.
The idea of the proof is to use the H\"older type continuity of $f$ to
show that one may actually treat $f$ locally as a constant in the relevant scales.
To show this, fix $p\in (1,\infty)$ and observe  that by Lemma \ref{Lp-multiscale} we have 
\begin{equation}\label{discret2}
\|f\otimes_\delta 1_{[0,1]^d}g\|_{L^p(\R^d)}\lesssim \|f\|_{L^p(\R^d)}\|1_{[0,1]^d}g\|_{L^p(\R^d)}.
\end{equation}
From Definition \ref{mtp}, for any $\delta >0$ we may decompose
$$f\otimes_\delta g (x)= \sum_{k\in\Z^d}\big(f(\cdot +k\delta)\otimes_\delta (1_{[0,1]^d}g(\cdot-k))\big)(x).$$
Hence
$$
H_\delta\coloneqq f\otimes_\delta (g-\widetilde g)=\sum_{k\in\Z^d}
(\Delta_{k\delta} f) \otimes_\delta 1_{[0,1]^d}g(\cdot-k).
$$
By the envelope property of $f$ there is  $\varepsilon\in (0,a-d)$ so that $\|\Delta_{k\delta} f\|_{L^p(\R^d)}\lesssim (|k|\delta)^\eps.$ Hence
an application of \eqref{discret2}  and our assumption on $g$ yield that
$$
\| H_\delta\|_{L^p(\R^d)}\lesssim \sum_{k\in\Z^d}(|k|\delta)^\eps
\|1_{[0,1]^d}g(\cdot -k)\|_{L^p(\R^d)}\lesssim \delta^\varepsilon,
$$
and the neglibility of $H_\delta$ follows.
\end{proof}

We remark that later on when we deal with stochastic multiscale functions then the natural analogue of the above lemma is no longer valid, causing some additional technical complications.

We now describe the weak convergence of multiscale functions in the limit $\delta \to 0$.

\begin{lemma}\label{weak} 
 Let $F = F_\delta$ be a multiscale function and $1 < p < \infty$.  Then $\|F_\delta\|_{L^p(\R^d)}$ is bounded uniformly in $\delta$.  Furthermore, there exists $F_0 \in L^p(\R^d)$ such that $F_\delta$ converges weakly in $L^p(\R^d)$ to $F_0$. Actually, 
 there is $\eps >0$ such that for any 
 test function $\phi\in C^\infty_0(\R^d)$
 $$
 \left|\int_{\R^d}(F_0(x)-F_\delta(x))\phi(x)dx\right|\lesssim_{\phi,F_\delta} \delta^\eps
 $$
\end{lemma}

\begin{proof}  By linearity it suffices to treat the cases when $F$ is either a multiscale tensor product or a negligible function.  The claims are trivial in the latter case, so assume that $F_\delta = f \otimes_\delta g$ for some envelope function $f$ and localized function $g$. 

The uniform boundedness of $\|F_\delta\|_{L^p(\R^d)}$ follows immediately from Lemma \ref{Lp-multiscale}.
In order to establish the weak convergence let us first consider the model case in which $g = 1_{[0,1]^d}$.  Then for a.e. $x$ we  have $F_\delta(x) = [f]_\delta(n)$, where $n$ is the integer part of $x/\delta$.  We thus see that
$$ F_\delta(x) - f(x) = \frac{1}{\delta^d} \int_{n\delta + [0,\delta]^d} f(y) - f(x)\ dy$$
and so by the triangle inequality
$$ |F_\delta(x) - f(x)| \leq \frac{1}{\delta^d} \int_{[-\delta,\delta]^d} |\Delta_r f(x)|\ dr.$$
Taking $L^p$ norms, applying Minkowski's inequality, and using the fact that $f \in \Lambda^{\eps_p,p}(\R^d)$ for some $\eps_p > 0$ we conclude that $F_\delta$ converges strongly in $L^p(\R^d)$ to $f$, which certainly suffices.  

By subtracting a constant multiple of this model case, we may assume in general that $g$ has mean zero.  We  claim that $F_\delta$ now converges weakly to zero.  Let $\phi \in C^\infty_0(\R^d)$ be a test function.  We need to show that
$$ \int_{\R^d} \sum_{n \in \Z^d} [f]_\delta(n) g_{[n,\delta]}(x) {\phi}(x)\ dx \to 0$$
as $\delta \to 0$.  Using the mean zero nature of $g$, we can rewrite the left-hand side as
$$ \delta^d \sum_{n \in \Z^d} [f]_\delta(n) \int_{\R^d} g(r) \Delta_{r\delta} {\phi}(n\delta)\ dr.$$
Since $\phi$ is a test function and $g$ is localized, the inner integral has magnitude $O( \delta )$, and furthermore vanishes unless $n= O_\phi(1/\delta)$.  
Thus the whole expression is bounded by $\delta\int_{|x|<c(\phi)/\delta} |f|$, and the Lemma follows.
\end{proof}

Parts (i) and (iii) of the  the following corollary follow immediately from the above proof, and (ii) is a consequence of (i). 

\begin{corollary}\label{envelope=multiscale} \text{}
\begin{itemize}
\item[(i)]
If  $f$ is an envelope function then $f-f \otimes_\delta 1_{[0,1]^d}$ is negligible.
\item[(ii)]
Every envelope function is a multiscale function. 
\item[(iii)] If $F_\delta$ is a multiscale function with expansion $F_\delta = \sum_{j=1}^J f_j \otimes_\delta g_j + G_\delta$ $($where $G_\delta$ is negligible$)$, then for any $p\in (1,\infty)$
$$
F_\delta \underset{\delta\to 0}{\longrightarrow}  \sum_{j=1}^Jc_jf_j
\quad  \textrm{\rm weakly in  } L^p\;\; \textrm{\rm with}\;\;
c_j\coloneqq  \int_{\R^d} g_j,\;\; j=1,\ldots, J.
$$
\end{itemize}
\end{corollary}
\noindent We remark that conclusion (iii) makes precise the intuitively obvious statement  that a multiscale tensor product approximates in some
natural sense (a multiple of) the envelope function as $\delta\to 0.$

The sum of two multiscale functions is clearly a multiscale function.  We proceed to give  other closure properties of multiscale functions, the first one being the closure under multiplication.

\begin{proposition}\label{multmult}  If $F = F_\delta$ and $G = G_\delta$ are multiscale functions, then $FG = F_\delta G_\delta$ is also a multiscale function.
\end{proposition}

\begin{proof}  If either $F$ or $G$ is negligible, then by Lemma \ref{weak} and H\"older's inequality we see that $FG$ is also negligible.  Hence by linearity we may assume that $F, G$ are multiscale tensor products, e.g., $F = f \otimes_\delta g$ and $G = f' \otimes_\delta g'$.  By Lemma \ref{adecay} we may assume in addition that supp$(g)\subset[0,1]^d$ and
supp$(g')\subset [0,1]^d. $ Then, by observing that $g_{[n,\delta]}, g'_{n',\delta}$ have disjoint supports if $n\not=n'$ we get
$$ FG(x) = \sum_{n \in \Z^d} [f]_\delta(n) [f']_\delta(n) g''_{[n,\delta]},$$
where $g'' \coloneqq gg'$ is a localized function.
 From Lemma \ref{commutator} 
we see that
$$ \sum_{n \in \Z^d} ([f]_\delta(n) [f']_\delta(n) - [ff']_\delta(n)) g''_{[n,\delta]}$$
is negligible. Hence $F G$  is equivalent to 
$$ \sum_{n \in \Z^d} [ff']_\delta(n) g''_{[n,\delta]}$$
which equals $(ff') \otimes_\delta g''(x)$.  Since $g''$ is localized, and (by Lemma \ref{envprod}) $ff'$ is an envelope function, the claim follows.
\end{proof}

\begin{corollary}\label{product_equivalence} Assume that $f,f'$ are envelope functions and $g,g'$ are localized functions. Then
the product $(f\otimes_\delta g)(f'\otimes_\delta g')$ is 
a multiscale function equivalent to either of the multiscale tensor products
$ff'\otimes \widetilde g_1$, $ff'\otimes \widetilde g_2$, where 
$$
\widetilde g_1(x) \coloneqq 1_{[0,1]^d}(x)\sum_{n,m\in \Z^d} g(n+x)g'(m+x)
$$
and 
$$
\widetilde g_2(x) \coloneqq \sum_{n\in \Z^d} g(n+x)g'(x).
$$
\end{corollary}
\begin{proof}
The statement concerning $\widetilde g_1$ follows directly from examining the proofs of Lemma  \ref{adecay} and Proposition  \ref{multmult}. The second statement
in turn follows  from  Lemma \ref{adecay}
by  observing that $\widetilde g_2$ is localized and $\widetilde g_1(x)=1_{[0,1]^d}(x)\sum_{k\in\Z^d}\widetilde g_2(x+k)$.
\end{proof}

Interestingly enough, the multiscale property is also preserved under (translation and scaling invariant) singular integrals. This is
not at all evident a priori since $Tg$ is usually not even integrable 
if $g$ is a localized function. Recall from  standard Calder\'on-Zygmund theory (see e.g., \cite{stein:small})  that $T$ extends to a bounded linear operator on $L^p(\R^d)$ for all $1 < p < \infty$.

\begin{proposition}\label{tf}  If $F = F_\delta$ is a multiscale function, and $T$ is a (translation and dilation invariant) singular integral operator (independent of $\delta$), then $TF = T F_\delta$ is also a multiscale function.
\end{proposition}

\begin{proof}  If $F$ is negligible, then $TF$ is negligible also since $T$ is bounded on every $L^p(\R^d)$ space.  So we may assume that $F_\delta = f \otimes_\delta g$ for some envelope function $f$ and localized function $g$.  To simplify the notation we now allow all implicit constants to depend on $f,g,T,d$.

First suppose that $g = 1_{[0,1]^d}$.  By Corollary \ref{envelope=multiscale} $F_\delta$ differs from $f$ by a negligible function, thus $TF_\delta$ differs from $Tf$ by a negligible function.  Since $f$ is an envelope function, and $T$ is translation-invariant and bounded on every $L^p(\R^d)$, we conclude that $Tf$ is an envelope function, and thus a multiscale function again by Corollary \ref{envelope=multiscale}, and the claim follows.

By linearity it now suffices to treat the case when $g$ has mean zero.  Using the translation and dilation invariance of $T$, we observe that
$$ TF = \sum_{n \in \Z^d} [f]_\delta(n) T(g_{[n,\delta]}) = \sum_{n \in \Z^d} [f]_\delta(n) (Tg)_{[n,\delta]}.$$
If $Tg$ were a localized function we would now be done, but this clearly not true in general. However,   $Tg(x)$  decays roughly like $\langle x\rangle^{-d-1}$ or, more precisely, we have for any $x \in \R^d$ and $1 < p < \infty$ that
\begin{equation}\label{tgp}
 \| Tg \|_{L^p(B(x,1))} \lesssim_{p} \langle x\rangle^{-d-1},
\end{equation}
whence Lemma \ref{adecay} applies and the desired conclusion follows.
In order to verify \ref{tgp}, observe first from the $L^p$ boundedness of $T$ that the claim is easy for $|x| \leq 4$, so we may assume $|x| > 4$. We then use the localized mean zero nature of $g$ to decompose $g$ into a piece $g_1$ of $L^p$ norm $O_{p}( \langle x \rangle^{-d-1} )$ supported in $\R^d\setminus B(0,|x|/8)$ and mean zero, and a mean zero localized function (with quantitative bounds independent of $x$) $g_2$ supported in  $B(0,|x|/4)$.  The contribution of $g_1$ is acceptable by the $L^p$ boundedness of $T$; the contribution of $g_2$ is acceptable by using the mean zero nature of $g_2$ to write for $x'\in B(x,1)$
$$ Tg_2(x') = \int_{B(0,|x|/4)} (K(x',y) - K(x',0)) g_2(y),$$
where $K(x,y) = \Omega(\frac{x-y}{|x-y|})/|x-y|^d$ is the singular kernel of $T$, and then using the triangle inequality, the Calder\'on-Zygmund type bounds on $K$,i.e.
$$
|K(x',y) - K(x',0)|\leq C|y||x'|^{-d-1}\quad \textrm{for}\quad |y|<2|x'|,
$$
 and the localized nature of $y\mapsto g_2(y)|y|$.
\end{proof}

\begin{corollary} For any  envelope function $f$, localized function
$g$ and singular integral $T$ the application $T(f\otimes_\delta g)$
is equivalent to the multiscale function $ATf+f\otimes_\delta g'$ where
$$
A\coloneqq \int_{\R^d}g\qquad \textrm{and}\qquad g'(x)\coloneqq  T\big(g-A1_{[0,1]^d}\big)(x).
$$
\end{corollary}

Iterating Proposition \ref{multmult} and Proposition \ref{tf}, we obtain our deterministic homogenization result for iterated singular integrals:

\begin{corollary}\label{singular_equivalence}  Let $\mu = \mu_\delta$ be a multiscale function, and let $T$ be a singular integral operator.  Let $m \geq 1$, let $1 < p < \infty$, and define $\mu_m = \mu_{m,\delta}$ recursively by $\mu_{1,\delta} \coloneqq  \mu_\delta$ and $\mu_{m,\delta} \coloneqq  \mu_\delta T \mu_{m-1,\delta}$ for $m > 1$.  Then $\mu_{m,\delta}$ is bounded in $L^p(\R^d)$ uniformly in $\delta$, and is weakly convergent to a limit $\mu_{m,0} \in L^p(\R^d)$.
\end{corollary}

\begin{remark} In principle it is possible to deduce a formula for
the limit $\mu_{m,0}$ by a repeated application of Lemma \ref{weak} and Corollaries
\ref{product_equivalence} and \ref{singular_equivalence}.

\end{remark}

\section{Stochastic multiscale functions}\label{random-sec}

\newcommand{\Otilde}{\widetilde\Omega}
\newcommand{\otilde}{\widetilde\omega}

We now turn to a generalization of the above theory, in which the localized functions $g$ are allowed to be stochastic.  We fix a probability space $\Omega$, and then define a product probability space $$
\Otilde\coloneqq \Omega^{\Z^d} \coloneqq  \{ (\omega_n)_{n \in \Z^d}: \omega_n \in \Omega \hbox{ for all } n \in \Z^d \}.
$$
We often write $\otilde\coloneqq (\omega_n)_{n\in\Z^d}$ for an element of  $\Otilde.$

\begin{definition}[Stochastic localized function]\label{sms}  A \emph{stochastic localized function} is a function $g\colon \R^d \times \Omega \to \C$ (not depending on the scale parameter $\delta$), where $\Omega$ is a probability space, such that for every $1 < p < \infty$ and $k > 0$, the function $\langle x \rangle^k g(x,\omega)$ lies in $L^p(\R^d \times \Omega)$.  We view $g$ as a random function from $\R^d$ to $\C$.
\end{definition}

\begin{definition}[Stochastic multiscale tensor product]\label{smtp}  Let $f\colon \R^d \to \C$ be an envelope function and $g\colon \R^d \times \Omega \to \C$ be a localized function.  We define the \emph{multiscale tensor product} $f \otimes_\delta g \colon \R^d \times \Otilde \to \C$ of $f$ and $g$ to be the function
\begin{equation}\label{smtp-def}
 f \otimes_\delta g(x, \otilde) \coloneqq  \sum_{n \in \Z^d} [f]_\delta(n) g_{[n,\delta]}(x,\omega_n)
 \end{equation}
where $g_{[n,\delta]}$ is defined by Definition \ref{Rescale}.
One can view $f \otimes_\delta g$ as a random function from $\R^d$ to $\C$.
\end{definition}

\begin{remark} In the above definition we of course could have instead
spoken of independent copies of random functions $g(\cdot,\cdot)$ without introducing the product space $\Otilde.$ However, the product space perhaps makes things slightly more transparent, at least for readers with little background in probability.

\end{remark}

\begin{definition}[Stochastic negligible function]  A function $F = F_\delta \colon \R^d \times \Otilde \to \C$ depending on the parameter $0 < \delta < 1$ is said to be \emph{negligible} if for every $1 < p < \infty$ there exists $\eps_p > 0$ and $C_p > 0$ such that 
$$ \|F_\delta \|_{L^p(\R^d \times \Otilde)} \leq C_p \delta^{\eps_p}$$
for all $0 < \delta < 1$.  We view $F_\delta$ as a random function from $\R^d$ to $\C$.
\end{definition}

\begin{definition}[Stochastic multiscale function]\label{smf}  A function $F = F_\delta \colon \R^d \times \tilde \Omega \to \C$ depending on the parameter $0 < \delta < 1$ is said to be a \emph{stochastic multiscale function} if we can write
$$ F_\delta = \sum_{j=1}^J f_j \otimes_\delta g_j + G_\delta$$
where $J \geq 1$ is an integer, $f_1,\ldots,f_J$ are envelope functions, $g_1,\ldots,g_J$ are stochastic localized functions, and $G = G_\delta$ is a stochastic negligible function.
\end{definition}

\noindent As in the previous section, we say that functions $F_\delta$ and $F'_\delta$ are
{\it equivalent} if their difference is a stochastic negligible function.

\begin{example}  For each $n \in \Z^d$, let $\varepsilon_n \in \{-1,1\}$ be an i.i.d. collection of signs, and let $Q$ be a cube in $\R^d$.  Then the random function
$$ F_\delta(x) \coloneqq  1_Q(x) \sum_{n \in \Z^d} \varepsilon_n 1_{n\delta + [0,\delta)^d}(x)$$
is a stochastic multiscale function (setting $\Omega = \{-1,1\}$ with the uniform distribution, and $\varepsilon_n$ to be the $n^{th}$ coordinate function of $\tilde \Omega = \Omega^{\Z^d}$).
\end{example}
 
\begin{remark} In the case when the probability space $\Omega$ is trivial, stochastic multiscale functions are essentially the same as deterministic multiscale functions.
\end{remark}

Lemma \ref{loc} and its proof  immediately  generalize, so that
we have
\begin{equation}\label{loc3}
\left\| \sum_{n \in\Z^d} f_n(x,\omega_n)\right\|_{L^p(\R^d\times\Otilde)} \lesssim_{p,d,N} (\sum_{n \in \Z^d} \| \langle x \rangle_{[n,\delta]}^d f_n(x,\omega) \|_{L^p(\R^d\times\Omega)}^p)^{1/p}.
\end{equation}
In a similar vein, Lemma
 \ref{Lp-multiscale} also generalizes to the stochastic setup, one just replaces
 $\|\langle \cdot \rangle^dg\|_{L^p(\R^d)}$ by $\|\langle \cdot \rangle^dg\|_{L^p(\R^d\times\Omega)}$  
on the right hand side.  In particular, we obtain 

\begin{lemma}\label{fbound}  Let $F = F_\delta \colon \R^d \times \tilde \Omega \to \C$ be a stochastic multiscale function, and let $1 < p < \infty$.  Then $\|F_\delta \|_{L^p(\R^d \times \tilde \Omega)}$ is bounded uniformly in $0 < \delta < 1$.
\end{lemma}

\begin{remark}\label{generalLp} Exactly  the same proof actually shows that 
$\|F_\delta \|_{L^p(\R^d \times \tilde \Omega)}$ stays  bounded in
$\delta$ for more general functions of the form
$$
F_\delta(x,\widetilde\omega)= \sum_{n \in \Z^d} [f]_\delta(n) (g_n)_{[n,\delta]}(x, \widetilde\omega)
$$
assuming only the uniform localization $\sup_{n\in\Z^d}\|\langle \cdot \rangle^N g_n(\cdot,\cdot)\|_{L^p(\R^d\times\Otilde)}<\infty$ for $N\geq 1$ and $p\in (1,\infty).$
\end{remark}

In turn,  Proposition \ref{multmult} has the following counterpart:

\begin{proposition}\label{multmult2}  If $F' = F'_\delta$ is a deterministic multiscale function and $F = F_\delta$ is a stochastic multiscale function, then $F'F = F'_\delta F_\delta$ is a stochastic multiscale function.
\end{proposition}
\begin{proof} It is enough to treat the case   $F=f\otimes_\delta g$ and $F'=f'\otimes_\delta g',$ where $f,f'$ are envelope functions, $g$
is a stochastic localized function and $g'$ is a deterministic localized function. Furthermore, by Lemma \ref{adecay} we may assume
that $g'$ has support in $[0,1]^d.$ 
Fix $p>1$ and observe that by the definition of localized functions  and H\"older's inequality we have 
\begin{equation}\label{gg}
\| \langle \cdot \rangle^d g'(\cdot-m) g (\cdot ,\cdot)\|_{L^p(\R^d\times \Otilde)}\lesssim_{g,g',p,a} \langle m \rangle^{-a}
\end{equation}
for all $a>1, $ in particular for $a=d+1$. It follows  that $\widetilde g$ is a stochastic localized function, where $\widetilde g$ is defined by $\widetilde g(x,\omega)\coloneqq \sum_{m\in\Z^k}g'(x-m)g(x,\omega)$. 
By Lemma \ref{envprod}, the proposition follows as soon as we show that 
$F_\delta' F_\delta$ is stochastically equivalent to $(ff')\otimes_\delta \widetilde g.$
Note that by Lemma \ref{commutator} and the stochastic counterpart of Lemma \ref {Lp-multiscale}  the latter function is equivalent to $H_\delta\coloneqq \sum_{n\in\Z^d}[f]_\delta(n)[f']_\delta (n)\widetilde g_{[n,\delta]}(\cdot,\omega_n)$, so it suffices to show that the difference $F_\delta'F_\delta-H_\delta$ is negligible. To that end we compute
\begin{align*}
(F_\delta F'_\delta-H_\delta)(x,\widetilde\omega) &=\sum_{m\in\Z^d}\left( \sum_{n\in\Z^d}[f]_\delta(n)\big(\Delta_m [f']_\delta(n)\big)g'_{[m+n,\delta](x)} g_{[n,\delta]}(x,\omega_n)\right)
\end{align*} 
Using \ref{gg},  first applying the stochastic version of Lemma \ref{loc}, and then Lemma \ref{discret} together with H\"older's inequality, we obtain
\begin{align*}
\|F_\delta F'_\delta-H_\delta\|_{L^p(\R^d\times\widetilde\Omega)} &\leq \delta^{d/p}\sum_{m\in\Z^d}\|[f]_\delta \Delta_m [f']_\delta\|_{\ell^p(\Z^d)} \langle m \rangle^{-d-1} \\
&\lesssim_{f,f',p,g,g'}\delta^{d/p}\sum_{m\in\Z^d} \langle m \rangle^{-d-1} |m\delta|^{\varepsilon_{2p}}\delta^{-d/2p}\delta^{-d/2p}\\
&\lesssim\delta^{\varepsilon_{2p}} .
\end{align*}
\end{proof}

We can now state our main technical result about stochastic multiscale functions: 

\begin{theorem}[Main estimate on stochastic multiscale functions]\label{mainthm} Let $m \geq 1$, let $\mu^{(1)} = \mu^{(1)}_\delta,\ldots,\mu^{(m)} = \mu^{(m)}_\delta$ be stochastic multiscale functions, and let $T_1,\ldots,T_{m-1}$ be singular integral operators.  Define $\mu_m = \mu_{m,\delta} \colon \R^d \times \tilde \Omega \to \C$ recursively by 
\begin{equation}\label{mum-def}
\mu_{1,\delta} \coloneqq  \mu_\delta^{(1)}; \quad \mu_{m,\delta} \coloneqq  \mu_\delta^{(m)} T_{m-1} \mu_{m-1,\delta}\;\; \hbox{ for } m > 1.
\end{equation}
Then $\mu_{m,\delta}$ is bounded in $L^p(\R^d \times \tilde \Omega)$ uniformly in $0 < \delta < 1$, for any $p\in (1,\infty).$  Furthermore, there exists a (deterministic) limit function $\mu_{m,0} \in L^p(\R^d)$ and $\eps > 0$ with the property that given any test function $\phi \in C^\infty_0(\R^d)$ we have
\begin{equation}\label{conc}
 \P( |\int_{\R^d} [\mu_{m,\delta}(x,\omega) - \mu_{m,0}(x)] {\phi(x)}\ dx| \geq \delta^{\eps} | )
\lesssim_{m,d,\eps,\mu^{(1)},\ldots,\mu^{(m)},T_1,\ldots,T_{m-1},\phi} \delta^\eps
\end{equation}
where $\P$ denotes the probability measure on $\tilde \Omega$.
\end{theorem}

The proof of this theorem is lengthy and shall occupy the next section.  Later, in Section \ref{se:qchomogenization} we shall give applications to random Beltrami equations through the following immediate consequence:

\begin{corollary}[Almost sure convergence]\label{co:main}
Let $\mu_{m,\delta}$ be as in Theorem \ref{mainthm} and assume that
$a\in (0,1).$ Then almost surely
$\mu_{m,a^{j}}\to \mu_{m,0}$ weakly in $L^p(\R^d)$ as $j\to\infty$, for all $p\in (1,\infty)$.

\end{corollary}
\begin{proof} The almost sure weak convergence, when tested against 
a test function follows immediately from Theorem \ref{mainthm}
and the Borel-Cantelli lemma. For the rest it is enough to recall
that the sequence $\mu_{m,a^{j}}$ is uniformly bounded in each 
$L^p(\R^d)$ and that one may pick a countable set of test functions that is dense in every $L^p(\R^d)$, with $1<p<\infty.$
\end{proof}

Theorem \ref{th:main22} then follows from Theorem \ref{mainthm} and Corollary \ref{co:main}.

We finish this section by observing that in the case where $g$ is bounded with  quick  decay as $x\to\infty$, our definition of a  deterministic multiscale function is equivalent to the product of
the envelope function and the `bump field defined via $g$'.
\begin{lemma}\label{prod3}
Assume that $g$  is bounded and has quick decay in the first variable:
$$
|g(x,\omega)|\leq C_N \langle x \rangle^{-N}\quad \textrm{for all}\;\; N\geq 1\;\;\textrm{and}\;\; x\in\R^d,\; \omega\in\Omega.
$$
Then for any envelope function $f$ the stochastic multiscale tensor product $f\otimes_\delta g$ is equivalent to the function $\displaystyle f(x)\Big(\sum_{n\in \Z^d}g_{[n,\delta]}(x,\omega_n)\Big).$
\end{lemma}
\begin{proof} 
Observe first that by the  decay assumption $g$ is a stochastic localized function and the bump field is uniformly
bounded:
$$
|\sum_{n\in \Z^d}g_{[n,\delta]}(x,\omega_n)|\leq C\quad \textrm{for all}\;\; x\in\R^d.
$$
Combining this with Corollary \ref{envelope=multiscale}(i) we
see that the product $\displaystyle f(x)\Big(\sum_{n\in \Z^d}g_{[n,\delta]}(x,\omega_n)\Big)$ is equivalent to
$\displaystyle f_\delta(x)\Big(\sum_{n\in \Z^d}g_{[n,\delta]}(x,\omega_n)\Big)$ ,
where $$
f_\delta\coloneqq f\otimes_\delta 1_{[0,1)^d}=\sum_{n\in \Z^d}
[f]_\delta(n) 1_{n\delta+[0,\delta)^d}.
$$
By using  for each $n\in\Z^d$  the trivial identities 
$[f]_\delta(n)=\sum_{k\in\Z^d}[f]_\delta(n) 1_{(n+k)\delta+[0,\delta)^d}$  and $f_\delta=\sum_{k\in \Z^d}
[f]_\delta(n+k) 1_{(n+k)\delta+[0,\delta)^d}$ we may use  \eqref{loc3} to estimate
for any $p\in (1,\infty)$ 
\begin{align*}
Q_p\coloneqq &\|f\otimes_\delta g - f_\delta\Big(\sum_{n\in \Z^d}g_{[n,\delta]}(x,\omega_n)\Big)\|_{L^p(\R^d\times \widetilde\Omega)}\\
=&\left\|\sum_{k\in \Z^d}\left(\sum_{n\in \Z^d}\big(
\Delta_k [f]_\delta (n)\big) 1_{(n+k)\delta+[0,\delta)^d}
(x)\big)g_{[n,\delta]}(x,\omega_n)\right)
\right\|_{L^p(\R^d\times \widetilde\Omega)}\\
\lesssim&\sum_{k\in \Z^d}\left(\sum_{n\in \Z^d}\big|
\Delta_k [f]_\delta (n)\big|^p\left\| 1_{(n+k)\delta+[0,\delta)^d}
(x)\langle x \rangle_{[n,\delta]}^d g_{[n,\delta]}(x,\omega_n)\right\|^p_{L^p(\R^d\times\widetilde\Omega)}\right)^{1/p}\\
\lesssim &
\sum_{k\in \Z^d}
\|\Delta_k [f]_\delta \|_{\ell^p(\Z^d)}
\| 1_{k\delta+[0,\delta)^d}
(x) \langle x \rangle_{[0,\delta]}^d g_{[0,\delta]}(x,\omega)
\|_{L^p(\R^d\times \Omega)}\  .
\end{align*}
By assumption $\| 1_{k\delta+[0,\delta)^d}
(x)\langle x \rangle_{[0,\delta]}^d g_{[0,\delta]}(x,\omega)
\|_{L^p(\R^d\times \Omega)}\lesssim \langle k \rangle^{-2d}\delta^{d/p}$,
and hence an application of Lemma \ref{discret} yields that
$$
Q_p\lesssim
\sum_{k\in \Z^d}\frac{(k\delta)^{\varepsilon_p}\delta^{-d/p}\delta^{d/p}}{\langle k \rangle^{2d}}\lesssim\delta^{\varepsilon_p},
$$
which completes the proof.
\end{proof}

\section{Proof of Theorem \ref{mainthm}}\label{sec:proof}

The purpose of this section is to prove Theorem \ref{mainthm}.  To abbreviate the notation we allow all implied constants to depend on $m, d, T_1,\ldots,T_{m-1}, \mu^{(1)},\ldots,\mu^{(m)},\phi$.

We first settle the claim that $\mu_{m,\delta}$ is uniformly bounded in $L^p(\R^d \times \widetilde\Omega)$, stating this result as a separate lemma:
\begin{lemma}\label{uniformLp} Define $\mu_{m,\delta}(\cdot, \omega)$ as in
Theorem \ref{mainthm}. Then $\|\mu_{m,\delta}\|_{L^p(\R^d \times \widetilde\Omega)}$
is uniformly bounded in $\delta >0.$
\end{lemma}
\begin{proof}
 Fix $\omega \in \tilde \Omega$.  Since $T_1,\ldots,T_{m-1}$ are bounded on $L^q(\R^d)$ for every $1 < q < \infty$, we see from H\"older's inequality and induction on $m$ that
\begin{equation}\label{integralholder}
 \| \mu_{m,\delta}(\cdot, \omega) \|_{L^p(\R^d)} \lesssim_p \prod_{j=1}^m \| \mu_{\delta}^{(j)}(\cdot, \omega) \|_{L^{mp}(\R^d)}.
\end{equation}
Namely, if this is true for the value $ m-1$, we choose $q>1$ with
$\frac{1}{q}+\frac{1}{mp}=\frac1p$ and use the $L^q$-boundedness of $T_{m-1}$
to estimate
\begin{align*}
 \| \mu_{m,\delta}(\cdot, \omega) \|_{L^p(\R^d)}
\leq &\|\mu^{(m)}_\delta (\cdot,\omega)\|_{L^{mp}(\R^d)}\|T_{m-1}\mu_{\delta,m-1} (\cdot,\omega)\|_{L^q(\R^d)}\\
 \lesssim &\|\mu^{(m)}_\delta (\cdot,\omega)\|_{L^{mp}(\R^d)}\|\mu_{\delta,m-1} (\cdot,\omega)\|_{L^q(\R^d)}\\
 \lesssim &\|\mu^{(m)}_\delta (\cdot,\omega)\|_{L^{mp}(\R^d)}
 \prod_{j=1}^{m-1} \|\mu^{(j)}_\delta (\cdot,\omega)\|_{L^{(m-1)q}(\R^d)},
\end{align*}
and as $(m-1)q=mp, $ the desired inequality \eqref{integralholder} with  index $m$ follows. 
Finally  Fubini's theorem and another application of H\"older's inequality yields
\begin{equation}\label{stochintegralholder}
 \| \mu_{m,\delta} \|_{L^p(\R^d \times \tilde \Omega)} \lesssim_p \prod_{j=1}^m \| \mu_{\delta}^{(j)} \|_{L^{mp}(\R^d  \times \tilde \Omega)}.
\end{equation}
The claim then follows from Lemma \ref{fbound}.
\end{proof}
\begin{remark}\label{generalLp2}
The above bound also holds if the stochastic multiscale functions $\mu_\delta^{(j)}$
are replaced by more general ones described in Remark \ref{generalLp}.
\end{remark}

Next, by multilinearity, we may assume that each of the stochastic multiscale functions $\mu^{(j)}_\delta$ is either a stochastic negligible function, or is a stochastic multiscale tensor product.  If at least one of the $\mu^{(j)}_\delta$ is stochastically negligible, we see by
\eqref{stochintegralholder} that for every $1 < p < \infty$ there exists $\eps > 0$ such that
$$ \| \mu_{m,\delta} \|_{L^p(\R^d \times \tilde \Omega)} \lesssim_{p,\eps} \delta^\eps.$$
By applying $p=2$ (say) and setting $\mu_{m,0} :\equiv 0$ we easily obtain the claim \eqref{mum-def}.

We may thus assume that for each $1 \leq j \leq m$ we have
\begin{equation}\label{muj-delta}
\mu^{(j)}_\delta = f_j \otimes_\delta g_j
\end{equation}
for some envelope functions $f_j$ and some stochastic localized functions $g_j$.  We allow all implied constants to depend on $f_j$ and $g_j$.

Next, we shall make the qualitative assumption that the envelope functions $f_j$ are compactly supported in $\R^d$.  This is purely in order to justify certain interchanges of summation, as now all sums in the multiscale functions are finite for a fixed $\delta>0$. At the very end of the proof we describe how this assumption can be dispensed with by a standard limit argument.

For the next reduction, we observe that it suffices to show for each $\phi \in C^\infty_0(\R^d)$ that there exists a limit $z=z_\phi \in \C$ such that
$$ \P\left( \left|\int_{\R^d} \mu_{m,\delta} \phi(x)\ dx - z\right| \geq \delta^{\eps} | \right)
\lesssim_{\eps} \delta^\eps$$
for some $\eps > 0$ independent of $\phi, \delta$.
Indeed, the map $\phi \mapsto z_\phi$ is then a continuous (by the uniform boundedness of  $\| \mu_{m,\delta}(\cdot, \omega) \|_{L^p(\R^d)}$), linear, and densely defined functional on $L^p(\R^d)$ for every $1 < p < \infty$, and can then be used to reconstruct $\mu_{m,0}$ by duality.

By \eqref{muj-delta} and \eqref{smtp-def} we can write
$$ \mu^{(j)}_\delta(x,\omega) = \sum_{n_j \in \Z^d} \mu^{(j)}_{\delta,n_j}(x,\omega_{n_j})$$
where
\begin{equation}\label{muj-nj}
\mu^{(j)}_{\delta,n_j}(x,\omega_{n_j}) \coloneqq  [f_{j}]_\delta(n_j) (g_j)_{\delta,n_j}( x, \omega_{n_j} ).
\end{equation}
We can therefore expand out the expression 
\begin{equation}\label{muph}
\int_{\R^d} \mu_{m,\delta}(x,\omega) {\phi}(x)\ dx
\end{equation}
using \eqref{mum-def} as
\begin{equation}\label{muph-2}
 \sum_{\vec n \in (\Z^d)^m} X_{\delta,\vec n}
\end{equation}
where $\vec n \coloneqq  (n_1,\ldots,n_m)$, and $X_{\delta,\vec n}$ is the complex-valued random variable
\begin{equation}\label{Xdef}
X_{\delta,\vec n} \coloneqq 
\int_{\R^d}
[\mu^{(m)}_{\delta,n_m}(\cdot,\omega_{n_m}) T_{m-1} \ldots T_1 \mu^{(1)}_{\delta,n_1}(\cdot, \omega_{n_1})](x) {\phi(x)}\ dx.
\end{equation}
Note that our qualitative hypotheses ensure that for each fixed $\delta$, only finitely many of the $X_{\delta,\vec n}$ are non-zero, and that each of the random variables $X_{\delta,\vec n}$ are bounded.

To obtain the concentration result \eqref{conc} we use Chebyshev's inequality.  From that inequality we see that it suffices to show a first moment estimate
\begin{equation}\label{first-mom}
 | \sum_{\vec n \in (\Z^d)^m} \E(X_{\delta,\vec n}) - z | \lesssim_\eps \delta^\eps
\end{equation}
together with a second moment estimate of the form
\begin{equation}\label{second-mom}
 \E \left| \sum_{\vec n \in (\Z^d)^m} X_{\delta,\vec n} - \E(X_{\delta,\vec n}) \right|^2  \lesssim_\eps \delta^\eps
\end{equation}
for some $\eps > 0$ (independent of $\delta$ and $\phi$).  (One can also control higher moments, but the second moment will suffice for our application.) 

\subsection{The second moment estimate} Let us first settle the second moment estimate \eqref{second-mom}.  We can expand the left-hand side as
$$ \sum_{\vec n, \vec n' \in (\Z^d)^m} \E( X_{\delta,\vec n} \overline{X_{\delta,\vec n'}} ) - \E( X_{\delta,\vec n} ) \overline{\E(X_{\delta,\vec n'})}.$$
Now observe from \eqref{Xdef} that $X_{\delta,\vec n}$ and $X_{\delta,\vec n'}$ are independent and the corresponding term in the above sum vanishes, unless we have $n_j = n'_{j'}$ for some $1 \leq j, j' \leq m$.  Thus by the triangle inequality and Cauchy-Schwarz we can estimate the previous expression by
$$ 2 \sum_{1 \leq j,j' \leq m}
\sum_{\vec n, \vec n' \in (\Z^d)^m: n_j = n'_{j'}} \E( |X_{\delta,\vec n}|^2 )^{1/2} \E( |X_{\delta,\vec n'}|^2 )^{1/2}.$$
It therefore suffices to establish an estimate of the form
\begin{equation}\label{sumxx}
 \sum_{\vec n, \vec n' \in (\Z^d)^m: n_j = n'_{j'}} \E( |X_{\delta,\vec n}|^2 )^{1/2} \E( |X_{\delta,\vec n'}|^2 )^{1/2}
\lesssim_\eps \delta^\eps
\end{equation}
for all $1 \leq j,j' \leq m$.

Fix $j,j'$.  We now pause to give a basic estimate on the size of each of the $X_{\delta,\vec n}$.  Define the kernel $K_0$ by setting
$$ K_0(n) \coloneqq  \frac{1}{\langle n\rangle^{d}}.$$

\begin{proposition}[Size estimate]\label{sizeestimate}  If $1 < p < \infty$ and $\vec n \in (\Z^d)^m$, then
$$ \E( |X_{\delta,\vec n}|^p )^{1/p} \lesssim_{p} \delta^d \langle \delta |n_m| \rangle^{-2d}
\left( \prod_{i=1}^m |[f_{i}]_\delta(n_i)| \right) \prod_{i=1}^{m-1} K_0(n_{i+1}-n_i).$$
\end{proposition}

For the proof of the proposition we shall need the following weighted version of the $L^p$ bounds for singular integrals.  
\begin{lemma}[Localized singular integral bounds]\label{lsib}  Let $T$ be a singular integral operator.  If $n, n' \in \Z^d$, $\delta > 0$, $1 < p < \infty$, and $N > d$, then we have the bound
$$ \| \langle \cdot \rangle_{n',\delta}^{-N} T f \|_{L^p(\R^d)} \lesssim_{T,p,d,N} K_0(n-n') \| \langle \cdot \rangle_{[n,\delta]}^N f \|_{L^p(\R^d)},$$
where $f$ is any function for which the right-hand side is finite.
\end{lemma}
\begin{proof} By scaling we may set $\delta = 1$.  We have
$$ \| T f \|_{L^p(B(n',1))} \lesssim_{T,p,d} K_0(n-n') \| f \|_{L^p(B(n,1))}$$
for all $n,n' \in \Z^d$ and all $f \in L^p(B(n,1))$ (extending $f$ by zero outside of this ball).
Namely, if $|n-n'| \geq 2$ then the claim follows simply by using the integral representation of $T$ and the triangle inequality (and H\"older's inequality).  If $|n-n'| < 2$ the claim instead follows by using the boundedness of $T$ on $L^p(\R^d)$.  It then follows that
\begin{align*} 
&\| \langle \cdot \rangle_{n',1}^{-N} T f \|_{L^p(\R^d)}\lesssim 
\sum_{k\in\Z^d}\langle k \rangle^{-N}\|Tf\|_{L^p(B(n'+k,1))}\\
\lesssim&
\sum_{k\in\Z^d}\langle k \rangle^{-N}\Big(\sum_{\ell\in\Z^d} \langle n'+k-n-\ell \rangle^{-d}\|f\|_{L^p(B(n+\ell,1))}\Big)\\
\lesssim &\Big(\sum_{k,\ell\in\Z^d}\langle k \rangle^{-N} \langle n'+k-n-\ell \rangle^{-d} \langle \ell \rangle^{-N}\Big) \| \langle \cdot \rangle_{[n,1]}^N f \|_{L^p(\R^d)}.
\end{align*}
This yields the stated estimate, since the last written sum is easily
estimated to be less than $O\left(\langle n-n'\rangle^{-d}\right)$ by considering separately
the case $ \max(|k|,|\ell|)\leq |n-n'|/4$ and its complement.
\end{proof}

\begin{proof}[Proof of Proposition \ref{sizeestimate}]  Pick $p\in(1,\infty)$ and fix $\omega\in\Otilde$. Denote
$$
g(\cdot,\omega)\coloneqq  \langle x \rangle_{n_m,\delta}^{3d}[\mu^{(m)}_{\delta,n_m}(\cdot,\omega_{n_m}) T_{m-1} \ldots T_1 \mu^{(1)}_{\delta,n_1}(\cdot, \omega_{n_1})](x) 
$$
so that we may write
\begin{equation}\label{term}
|X_{\delta, \vec n}|=\Big|\int_{\R^d} g(x,\omega)\langle x \rangle_{n_m,\delta}^{-N}\phi(x)dx \Big|
\leq \|g(\cdot,\omega)\|_{L^p(\R^d)}\|\langle \cdot \rangle_{n_m,\delta}^{-3d}\phi \|_{L^{p'}(\R^d)}.
\end{equation}
By an inductive application of Lemma \ref{lsib} and H\"older's inequality as in the proof of \eqref{integralholder}, we see that
\begin{equation}\label{mu_prod}
\|g(\cdot,\omega)\|_{L^p(\R^d)}\lesssim 
\prod_{i=1}^{m-1} K_0(n_{i+1}-n_i)
\prod_{i=1}^m \| \langle \cdot \rangle_{n_i,\delta}^{6d} \mu^{(i)}_{\delta,n_i}(\cdot,\omega_{n_i}) \|_{L^{pm}(\R^d)}.
\end{equation}
Since $\phi$ is a Schwartz function one easily verifies that
$$
\|\langle \cdot \rangle_{n_m,\delta}^{-3d}\phi \|_{L^{p'}(\R^d)}\lesssim\frac{\delta^{d/p'}}{\langle \delta n_m \rangle^{2d}}.
$$
Moreover, \eqref{muj-nj} and the localized nature of $g_i$ yield that 
$$  \big(\E \| \langle \cdot \rangle_{n_i,\delta}^{6d} \mu^{(i)}_{\delta,n_i}(\cdot,\omega_{n_i}) \|_{L^{pm}(\R^d)}^{pm} \big)^{1/pm}
\lesssim \delta^{d/mp}|[f_{i}]_\delta(n_i)|.
$$
By combining these estimates with \eqref{term} the desired estimate follows by H\"older's inequality and the relation $1/p+1/p'=1$.
\end{proof}

In order to utilize  the above proposition we need to introduce   discrete fractional integrals.
To that end, given any real number $\alpha\in [0,d)$, define the more general kernels $K_\alpha \colon \Z^d \to \R^+$ on the integer lattice by 
\begin{equation}\label{K-def} K_\alpha(n) \coloneqq  \frac{1}{\langle n \rangle^{d-\alpha}}.
\end{equation}
The convolution of functions defined on the lattice $\Z^d $ is defined in the usual manner:
$$ F*G(n) \coloneqq  \sum_{m \in \Z^d} F(m) G(n-m).$$
By direct computation we have the convolution estimate
\begin{equation}\label{k-conv}
K_\alpha * K_\beta \lesssim_{\alpha,\beta,n} K_{\alpha + \beta}
\end{equation}
whenever $\alpha, \beta >0$ and $\alpha + \beta < d$.  These estimates are unfortunately not true at the endpoints $\alpha=0$ or $\beta =0$, due to the logarithmic failure of summability of $K_0$.  However, from Young's inequality we easily see that
\begin{equation}\label{young}
\| K_0 * f \|_{l^q(\Z^d)} \lesssim_{d,p,q} \|f\|_{\ell^p(\Z^d)}
\end{equation}
for all $1 \leq p < q \leq \infty$ and all $f \in l^p(\Z^d)$.

Finally, we are ready to estimate the left-hand side of \eqref{sumxx} by
\begin{align}
\lesssim 
S&\;\coloneqq \;\delta^{2d}\sum_{\vec n, \vec n' \in (\Z^d)^m: n_j = n'_{j'}} 
\langle \delta n_m \rangle^{-2d} ( \prod_{i=1}^m |[f_{i}]_{\delta}(n_i)| ) \prod_{i=1}^{m-1} K_0(n_{i+1}-n_i)\label{product2}\\
&\quad \times \langle \delta n'_m \rangle^{-2d} ( \prod_{i'=1}^m |[f_{i'}]_{\delta}(n'_{i'})| ) \prod_{i'=1}^{m-1} K_0(n'_{i'+1}-n'_{i'}).\nonumber
\end{align}
Writing $n_j = n'_{j'} = n$, we can rewrite this expression using the convolution operator $T_{K_0} f \coloneqq  f*K_0$  and by denoting $\Phi_\delta(n) \coloneqq  \langle \delta n \rangle^{-2d}$ as
\begin{align}\label{6product}
\sum_{n \in \Z^d} & |[f_{j}]_{\delta}(n)| |[f_{j'}]_{\delta}(n)|H_{1,\delta}(n)H_{2,\delta}(n)G_{1,\delta}(n)G_{2,\delta}(n)
\end{align}
with\footnote{Below the definitions of $H_{1,\delta}$ and its analogues are to be interpreted as follows: 
 starting from the right, one alternatively  performs either a pointwise
multiplication a sequence $|[f_{k}]_\delta|$ or an application by the
operator $T_{K_0}$.}
\begin{align*}
&H_{1,\delta}(n)\coloneqq \big(T_{K_0} (|[f_{j+1}]_\delta|) \ldots T_{K_0} (|[f_m]_\delta| \Phi_\delta)\big)(n) \\
&H_{2,\delta}(n)\coloneqq \big(T_{K_0} (|[f_{j'+1}]_\delta|)  \ldots T_{K_0} (|[f_m]_\delta| \Phi_\delta)\big)(n) \\
&G_{1,\delta}(n)\coloneqq \big( T_{K_0} (|[f_{j-1}]_\delta|)  \ldots T_{K_0} (|[f_1]_\delta|)\big)(n) \\
&G_{2,\delta}(n)\coloneqq \big( T_{K_0} (|[f_{j'-1}]_\delta|) \ldots T_{K_0} (|[f_1]_\delta|)\big)(n).
\end{align*}
In order to bound these functions, observe first that  for given $p>1$,
for any  $\widetilde\varepsilon>0$, and for  an arbitrary sequence $(a(n))_{n\in\Z^d}$   
\begin{align}\label{closetoinfinity}
\|a[f_i]_\delta\|_{\ell^p(\Z^d)}\leq \|a\|_{\ell^{p}(\Z^d)}\|[f_i]_\delta\|_{\ell^{\infty}(\Z^d)}\lesssim_{f_i,p,\widetilde\varepsilon}\delta^{-\widetilde\varepsilon}\|a\|_{\ell^{p}(\Z^d)},
\end{align} 
since Lemma \ref{discret} yields that $ \|[f_i]_\delta\|_{\ell^{\infty}(\Z^d)} \leq \|[f_i]_\delta\|_{\ell^{q}(\Z^d)}\lesssim \delta^{-d/q}$ for all $q>1$ and we just take $q$ large enough.
Fix $\varepsilon >0$. Using alternately the above estimate (with a very small value of $\widetilde\varepsilon$) and the boundedness of $T_{K_0}\colon \ell^p(\Z^d)\to \ell^q(\Z^d)$ for any $1<p<q<\infty$  we obtain that
\begin{equation}\label{H}
\|H_{k,\delta}\|_{\ell^{2+\varepsilon}(\Z^d)}\lesssim \delta^{-\varepsilon}
\|\Phi_\delta\|_{\ell^2(\Z^d)}\lesssim  \delta^{-\varepsilon-d/2},\quad k=1,2.
\end{equation}
Set $q=q(\varepsilon)=4\varepsilon^{-1}(2+\varepsilon)$  so that $2/q+1/(2+\varepsilon)=1/2,$ and  use \eqref{closetoinfinity} to similarly obtain the estimate
\begin{equation}\label{G}
\|G_{k,\delta}\|_{\ell^q(\Z^d)}\lesssim \delta^{-\varepsilon}
\|[f_1]_\delta\|_{\ell^{q-\eps'}(\Z^d)}\lesssim  \delta^{-\varepsilon-d/(q-\varepsilon')}\lesssim  \delta^{-(d+1)\varepsilon} ,\quad k=1,2,
\end{equation}
where we just picked $\varepsilon'>0$ small enough.
Finally, plugging the above bounds in \eqref{6product},  using $2/(1+\varepsilon)+4/q=1$ and the fact that $\|[f_1]_\delta\|_{\ell^{q}(\Z^d)}\lesssim \delta^{-d\varepsilon}$ we obtain via H\"older's inequality
\begin{align*}
S\;\leq \;&\delta^{2d}\|H_{1,\delta}\|_{\ell^{2+\varepsilon}(\Z^d)}\|H_{2,\delta}\|_{\ell^{2+\varepsilon}(\Z^d)}\|G_{1,\delta}\|_{\ell^{q}(\Z^d)}
\|G_{2,\delta}\|_{\ell^{q}(\Z^d)}\|[f_j]_\delta\|_{\ell^{q}(\Z^d)}
\|[f_{j'}]_\delta\|_{\ell^{q}(\Z^d)}\\
\lesssim\;&\delta^{2d}\delta^{-\varepsilon-d/2}\delta^{-\varepsilon-d/2}\delta^{-2(d+1)\varepsilon}\delta^{-2d\varepsilon}\\
\lesssim\;& \delta^{d-O(\varepsilon)}.
\end{align*}
If $j=1$ (resp. $j'=1$), the term $G_{1,\delta}$  (resp. $G_{2,\delta}$) is not present in \eqref{6product},
and the above argument goes through with obvious modifications.
The desired estimate follows as $\varepsilon >0$ is arbitrary.

\begin{remark}
One way to understand the obtained bound for the second moment is to observe that a
computation analoguous to the above one could also be used e.g., to estimate the quantity
$\E \| X_{\delta, \vec n, c}\|^2$, which we know to be bounded. However,
direct implementation the above method would give us a divergent upper bound
of the form $ \delta^{-O(\varepsilon)}$, due to the logarithmic non-boundedness
of the kernel $K_0$ on $\ell^p$-spaces, as we are  ignoring nontrivial cancellations that are behind  Proposition \ref{sizeestimate}.
Roughly speaking, what saves us above is that the condition $n_j=n'_{j'}$, due to independence, reduces the number  of
terms by a factor $\delta^{d}.$
\end{remark}
 
 \subsection{The first moment estimate}
It now remains to establish the first moment estimate \eqref{first-mom}, whose proof is more combinatorial in nature.  
We can split the left-hand side into finitely many components, depending on the equivalence class that $n_1,\ldots,n_m$ generates.  Given any surjective coloring function $c\colon \{1,\ldots,m\} \to \{1,\ldots,k\}$ which assigns a ``color'' in some finite set of integers $\{1,\ldots,k\}$ to every integer $\{1,\ldots,m\}$, let $(\Z^d)^m_c$ denote the set of all $\vec n \in (\Z^d)^m$ such that $n_j = n_{j'}$ if and only if $c(j) = c(j')$.  Clearly we can partition $(\Z^d)^m$ into finitely many of the $(\Z^d)^m_c$.  Thus it will suffice to show that for each coloring function $c$ there exists a complex number $z_c$ (independent of $\delta$, but depending on all other parameters) for which we have
$$ | \sum_{\vec n \in (\Z^d)^m_c} \E(X_{\delta,\vec n}) - z_c | \lesssim_\eps \delta^\eps.$$
Fix $c$.  We can reparameterise this as
$$
| \sum_{\vec n \in (\Z^d)^k_{{\neq}}} \E( X_{\delta,\vec n, c}) - z_c | \lesssim_\eps \delta^\eps.
$$
where $(\Z^d)^k_{\neq}$ is the space of all $k$-tuples $(n_1,\ldots,n_k) \in (\Z^d)^k$ with $n_1,\ldots,n_k$ distinct, and $X_{\delta,\vec n, c}\colon \Omega^k \to \C$ is the complex-valued random variable
$$
X_{\delta,\vec n,c} \coloneqq 
\int_{\R^d}
[\mu^{(m)}_{\delta,n_{c(m)}}(\cdot,\omega_{n_{c(m)}}) T_{m-1} \ldots T_1 \mu^{(1)}_{\delta,n_{c(1)}}(\cdot, \omega_{n_{c(1)}})](x) {\phi(x)}\ dx.
$$
Observe from the inclusion-exclusion principle that the sum $\sum_{\vec n \in (\Z^d)^k \backslash (\Z^d)^k_{\neq}} \E( X_{\vec n, c} )$ can be expressed as a finite linear combination of expressions of the form $\sum_{\vec n \in (\Z^d)^{k'}} \E_c( X_{\vec n, c'} )$ where $k' < k$ and $c'\colon \{1,\ldots,n\} \to \{1,\ldots,k'\}$ is a surjective coloring, and $c$ is a refinement of $c'$ (i.e., $c(j_1)=c(j_2)$ implies that $c'(j_1)=c'(j_2)$).  Thus by induction on $k$ it in fact suffices to show that for every pair of colorings $(c,c')$ with $c$ finer than $c'$ there exists a complex number $z'_{c,c'}$ for which we have
$$
| \sum_{\vec n \in (\Z^d)^k} \E_c(X_{\delta,\vec n, c'}) - z'_{c,c'} | \lesssim_\eps \delta^\eps,
$$
where we used the notation
$$
\E_cX_{\delta,\vec n,c'} \coloneqq 
\int_{\R^d}
\E[\mu^{(m)}_{\delta,n_{c'(m)}}(\cdot,\omega_{c(m)}) T_{m-1} \ldots T_1 \mu^{(1)}_{\delta,n_{c'(1)}}(\cdot, \omega_{c(1)})](x) \phi(x)\ dx.
$$
Let us now use Fubini's theorem to write
$$ \sum_{\vec n \in (\Z^d)^k} \E_c( X_{\vec n, c'} ) = \int_{\R^d} T_{c,c',\delta}(1)( x ) {\phi(x)}\ dx,$$
where $T_{c,c',\delta}$ is the (deterministic) operator
$$ T_{c,c',\delta}h(x) \coloneqq  \sum_{\vec n \in (\Z^d)^k} \E [\mu^{(m)}_{\delta,n_{c'(m)}}(\cdot,\omega_{c(m)}) T_{m-1} \ldots T_1 \mu^{(1)}_{\delta,n_{c'(1)}}(\cdot, \omega_{c(1)}) h](x).$$

Let us next verify the uniform boundedness of our `colored'
sum:
\begin{lemma}\label{uniLP} Assume that $h=h_\delta$ is a deterministic multiscale function. Then
\begin{equation}\label{Lph}
\|T_{c,c',\delta}h_\delta\|_{L^p(\R^d)}\leq C<\infty\quad \textrm{for}\;\; \delta >0.
\end{equation}
\end{lemma}
\begin{proof}
We double the number of coordinates in our probability space and consider the product (probability) space $\widetilde \Omega\times  \widetilde \Omega'$ whose
elements we can write as sequences $(\widetilde\omega,\widetilde\omega')=(\omega_n,\omega'_n)_{n\in Z^d}$, and choose unimodular random variables $Y_{k,j}\colon \widetilde\Omega\to \{ 1,-1\}$ 
for $k=1,\ldots ,m$ and $j\in\Z^d$ such that
$
\E Y_{1,n_1}\cdot Y_{m,n_m}
$
is equal to 1 if $(n_1,\ldots ,n_m)$ respects the coloring $c'$ (i.e., $n_\ell=n_{\ell'}$
for those
$\ell,\ell'\in\{ 1,\ldots , m\}$ that have the same color with respect to
$c'$), and otherwise this expectation is zero. For example, in the case
$m=2$ and one color (i.e., $c'(1)=c'(2)=1$) one may take $Y_{1,j}=Y_{2,j}=\Theta_j,$
where $(\Theta_j)$ is a Bernoulli sequence. In the general
case one associates independent copies of such sequences for all
the pairs $(k,k')$ that have the same color. More explicitly, one can set $\widetilde\Omega'\coloneqq \{-1,1\}^A$ with the Bernoulli measure where $A$ is the set of triples $(n,\ell,\ell')$ with $n\in\Z^d$ and $\ell,\ell' \in \{1,\ldots ,m\}$
with $c'(\ell)=c'(\ell')$, and set
$$
Y_{r,n}(\widetilde\omega')=\prod_{(n,\ell,\ell')\in A:\; r\in\{\ell,\ell'\}}\widetilde\omega'_{n,\ell,\ell'}
$$
for any $\widetilde\omega'=\widetilde\omega'_{(n,\ell,\ell')\in A},$  $r\in\{1,\ldots, m\},$ and $n\in \Z^d.$

 We may then write
\begin{equation}\label{representation}
T_{c,c',\delta}h_\delta(x)=\E_{\widetilde\Omega\times \widetilde \Omega'}\Big(\sum_{\vec n \in (\Z^d)^m}  [\widetilde \mu^{(m)}_{\delta,n_m}(\cdot,\omega_{c(m)},\widetilde \omega') T_{m-1} \ldots T_1 \widetilde\mu^{(1)}_{\delta,n_{1}}(\cdot, \omega_{c(1)},\widetilde \omega') h](x),\Big)
\end{equation}
where for $k\in\{ 1,\ldots ,m\}$ and $n\in\Z^d$ we set
$$
\widetilde\mu^{(k)}_{\delta,n}(x,\widetilde \omega,\widetilde\omega')\coloneqq 
\mu^{(k)}_{\delta,n}(x,\widetilde \omega)Y_{k,n}(\widetilde\omega').
$$
In particular, we may write
\begin{equation}\label{representation2}
T_{c,c',\delta}h=\E_{\widetilde\Omega\times \widetilde \Omega'}
H^{(m)}_\delta T_{m-1}\ldots H^{(1)}_\delta h_\delta,
\end{equation}
with
$$
H^{(k)}_\delta(x,\widetilde \omega,\widetilde\omega')=
\sum_{n\in \Z^d}[f_k]_\delta (n)(g_k)_{[n,\delta]}(x,\omega_{c(k)})Y_{k,n}(\widetilde\omega)
$$
Recalling Remark \ref{generalLp}, the argument of Lemma \ref{fbound}  applies as before since  the additional factors $Y_{k,n}$ or
having the variable $\omega_{c(k)}$ instead of $\omega_{n}$ do
not affect our old estimates, whence
$$
\|H^{(k)}_\delta\|_{L^p(\R^d\times \widetilde\Omega
\times \widetilde\Omega)} \leq_{p} C \qquad \textrm{for all}\; \delta >0, \;\; p\in (1,\infty).
$$
Finally, Lemma \ref{uniformLp} (together with Remark \ref{generalLp2}) and H\"older's inequality yields the desired result.
\end{proof}

We pause to clarify by an example the role of colorings introduced above.

\begin{remark}\label{re:splitting} In order to illustrate the use of
the colorings and the division to cases `split and `non-split'
(the latter notions will introduced shortly below in the proof of Proposition \ref{mainprop}) let us consider in case $m=3$ the expectation
$$
S\coloneqq \E \left(\sum_{(n_1,n_2,n_3)\in\Z^3} X_{n_1}TY_{n_2}TZ_{n_3}\right),
$$
that is of the type we have to handle.
Here the $X_n=X_n(x,U_n)$,  $Y_j=Y_n(x,U_n),$ $Z_n=Z_n(x,U_n)$ ($n\in\Z$) are (say bounded) random functions, and the  $U_j$ are i.i.d random variables. The linear operator $T$ could be e.g., a singular integral operator. In the first step one uses independence and Fubini to write the above sum in the form (the extra subindex $\not=$ indicates that one sums only over triples or tuples consisting of \emph{unequal} indices)
\begin{equation}\label{eq:isojako}
\begin{split}
S&=\sum_{n_1,n_2,n_3,\not=} (\E X_{n_1}) T(\E Y_{n_2})T(\E Z_{n_3})+
\sum_{n_1,n_2\not=} \big(\E (X_{n_1}TY_{n_1})\big)T\E Z_{n_2}\\
&+
\sum_{n_1,n_2\not=} \E \big( X_{n_1}T(\E Y_{n_2})T Z_{n_1}\big)+
\sum_{n_1,n_2\not=} (\E X_{n_1})T\E (Y_{n_2} T Z_{n_2})+
\sum_{n_1} \E \big(X_{n_1}TY_{n_1}TZ_{n_1}\big)
\\
&=: S_1+S_2+S_3+S_4+S_5.
\end{split}
\end{equation}
In the next step uses the inclusion exclusion principle to rewrite the sums so that one sums over all indices. For example, we obtain
\begin{eqnarray*}
S_1&=&\sum_{n_1,n_2,n_3} (\E X_{n_1})T (\E Y_{n_2})T(\E Z_{n_3})-
\sum_{n_1,n_2} (\E X_{n_1})T(\E Y_{n_1})T(\E Z_{n_2})
-\sum_{n_1,n_2} (\E X_{n_1})T(\E Y_{n_2})T(\E Z_{n_1})\\
&-&\sum_{n_1,n_2} (\E X_{n_1})T(\E Y_{n_2})T(\E Z_{n_2})
+2\sum_{n_1} (\E X_{n_1})T(\E Y_{n_1})T(\E Z_{n_1})\\
&\coloneqq &S_{11}-S_{12}-S_{13}
-S_{14}+2S_{15}.
\end{eqnarray*}
\end{remark}
Each of these terms can be expressed via a pair of colourings $(c,c')$,
let $c_{\ell k}$ and $c'_{\ell k}$ stand for the colours of the term $S_{\ell k}.$
At most three colours are needed.
 We have $c_{1\ell}=(1,2,3)$ for each $\ell\in\{1,\ldots 5\}$.
 In turn, $c'_{11}=(1,2,3)$, $c'_{12}=(1,1,2)$,$c'_{13}=(1,2,1)$,
 $c'_{11}=(1,2,2)$, and $c'_{15}=(1,1,1)$.

 In  similar vein the term $S_2$ can be rewritten as
\begin{eqnarray*}
S_2&=&\sum_{n_1,n_2} \big(\E (X_{n_1}TY_{n_1})\big)T\E Z_{n_2}-
\sum_{n_1} \big(\E (X_{n_1}TY_{n_1})\big)T\E Z_{n_1}\\
&=:& S_{21}-S_{22}.
\end{eqnarray*}
Now the colourings are $c_{21}=c_{22}=(1,1,2)$,
$c'_{21}=(1,1,2)$ and $c'_{22}=(1,1,1)$. The terms $S_3$ and
$S_4$ are analogous, and finally the term $S_5$ needs no
further subdivision and one has $c_5=c'_5=(1,1,1).$

Among the  terms $ S_{11},S_{12},S_{13},S_{14},S_{15},S_{21},
S_{22}$ and $S_5$ the terms $S_{11}, S_{12},S_{14}$ and $S_{21}$
will later on be designated as \emph{split}, and the remaining ones as \emph{nonsplit}. This means the following: for a split term  one can concretely divide the defining sum  to independent left side and right hand side summations, and
 also the expectations split accordingly.  E.g., we may write
$$
S_{21}= fTg\qquad \textrm{with}\quad f\coloneqq  \quad \sum^n_{j_1} \E (X_{j_1}TY_{j_1})
\quad \textrm{and} \quad g\coloneqq  \sum_{j_2}^n \E Z_{j_2}.
$$

\bigskip

We return to the main course of the argument and note that,
in view of Lemma \ref{weak}, it  suffices to show that
\begin{proposition}[Main proposition]\label{mainprop}  If $c\colon \{1,\ldots,m\} \to \{1,\ldots,k\}$ is surjective, and $h = h_\delta$ is a (deterministic) multiscale function, then $T_{c,c'}(h) = T_{c,c',\delta}(h_\delta)$ is also a (deterministic) multiscale function.
\end{proposition}

The remainder of this section is devoted to the proof of this proposition.

We first observe that one proves easily (e.g., compare the proof of  Proposition \ref{multmult2}) that if we  know the claim (for a given colouring $c'$) in the special case $h_\delta = 1$, then it is true (for the given colouring $c'$) in the general case. Namely, the proof of  Proposition \ref{multmult2} applies as such to the product term $H^{(1)}_\delta h_\delta$
in  representation \eqref{representation2} verifying that it can be replaced by $\widetilde H^{(1)},$ which is of the same form as $H^{(1)}$, and by decoupling the representation we obtain an
expression with $1$ in place of $h_\delta.$

We induct on $k$, i.e., the number of colors in $c'$.  If there is only one color in $c'$, then 
\begin{align*}
 T_{c,c',\delta}(1)(x) \coloneqq  &\sum_{ n \in (\Z^d)} \E_c [\mu^{(m)}_{\delta,n}(\cdot,\omega_{c(m)}) T_{m-1} \ldots T_1 \mu^{(1)}_{\delta,n}(\cdot, \omega_{c(1)}) ](x)\\
 =&\Big[\sum_{n \in \Z^d} \Big(\prod_{j=1}^m[f_j]_\delta(n)\Big)
  g_{[n,\delta]}(x)\Big],
\end{align*}
where
$$
g(\cdot,\omega)=:E g_m(\cdot ,\omega_{c(j)})T_{m-1}\ldots T_1  g_1(\cdot,\omega_{c(j)}) 
$$
Obviously $g$ is a localized function, and 
hence Lemma \ref{commutator} and Proposition  \ref{multmult} verify that $T_{c,c',\delta}(h)$ is a multiscale function.  Now we suppose inductively that $k> 1$, and that the claim has already been proven for all smaller values of $k$.

We begin by disposing of the \emph{split} case, in which there exists a non-trivial partition $\{1,\ldots,m\} = \{1,\ldots,j\} \cup \{j+1,\ldots,m\}$ with $1 \leq j < m$ such that $c'(\{1,\ldots,j\})$ and $c'(\{j+1,\ldots,m\})$ are disjoint.
By relabeling colors if necessary we may assume that $c'(\{1,\ldots,m\}) = \{1,\ldots,k'\}$ for some $1 \leq k' < k$. Then,  we let $c_1'\colon \{1,\ldots,j\} \to \{1,\ldots,k'\}$ be the restriction of $c$ to $\{1,\ldots,j\}$, and $c'_2\colon \{1,\ldots,m-j\} \to \{1,\ldots,k-k'\}$ be the function $c'_2(i) \coloneqq  c'(i+j)-k'$. The restrictions 
$c_1,c_2$ are defined analoguously using the fact that $c$ refines $c'.$ Observe by the definition of $\E_c$ that
\begin{align*} &T_{c,c',\delta}(1)(x)\\ \coloneqq  &\sum_{\vec n \in (\Z^d)^{k-k'}} \E_{c_2} [\mu^{(m)}_{\delta,n_{c'_2(m-j)}}(\cdot,\omega_{c_2(m-j)}) T_{m-1} \ldots T_{j+1} \mu^{(j+1)}_{\delta,n_{c'_2(1)}}(\cdot, \omega_{c_2(1)}) T_j T_{c_1,c'_1,\delta}(1)](x).
\end{align*}
By induction hypothesis, $T_{c_1,c'_1,\delta}(1)$ is a deterministic multiscale function, and then by Proposition \ref{tf}
$T_j T_{c_1,c'_1,\delta}(1)$ is also.  The claim then follows by another application of the inductive hypothesis.

Finally, we deal with the more difficult \emph{non-split} case in which no non-trivial partition of the above type exists.  
In other words, we need to show that
$$ T_{c,c'\delta}(1)(x) = \sum_{\vec n \in (\Z^d)^k} \E \left(\mu^{(m)}_{\delta,n_{c'(m)}}(\cdot,\omega_{c(m)}) T_{m-1} \ldots T_1 \mu^{(1)}_{\delta,n_{c'(1)}}(\cdot, \omega_{c(1)})\right)(x)$$
is a multiscale function.  Using \eqref{muj-nj} and the fact that all the $T_1,\ldots,T_{m-1}$ commute with dilations,
we can rewrite $T_{c,c',\delta}1(x)$ as
$$ T_{c,c'\delta}1(x) \coloneqq  \sum_{\vec n \in (\Z^d)^k} (\prod_{i=1}^m [f_i]_\delta(n_{c'(i)})) (G_{\vec n})_{[0,\delta]}(x)$$
where
$$ G_{\vec n}(x) \coloneqq  
\E\left(g_m( \cdot - n_{c'(m)}, \omega_{n_{c(m)}} ) T_{m-1} \ldots T_1 g_1(\cdot - n_{c'(1)}, \omega_{n_{c(1)}})\right)(x).$$
Using the translation-invariance of the $T_1,\ldots,T_{m-1}$, we can rewrite this as
$$
 T_{c,c',\delta}1(x) \coloneqq  \sum_{n \in \Z^d} \sum_{\vec r \in (\Z^d)^k: r_{c(m)} = 0}
 (\prod_{i=1}^m [f_i]_\delta(n + r_{c(i)})) (G_{\vec r})_{[n,\delta]}(x).$$

To estimate this expression, we observe that exactly as in \eqref{mu_prod} we have
for any $\vec r \in (\Z^d)^k$, $N > 0$, and $1 < p < \infty$ the estimate
\begin{equation}\label{warp}
 \| \langle \cdot \rangle^N G_{\vec r} \|_{L^p(\R^d)} \lesssim_{N,p} \prod_{i=1}^{m-1} K_0( r_{c(i+1)} - r_{c(i)} )
 \end{equation}
We combine this lemma with the non-split nature of $c$ to obtain the following.
\begin{lemma}\label{rcm}  For any $N > 0$ and $1 < p < \infty$ there exists $\alpha > 0$ such that
$$ \| \langle \cdot \rangle^N \sum_{\vec r \in (\Z^d)^k: r_{c(m)} = 0: R \leq \langle \vec r \rangle < 2R} |G_{\vec r}| \|_{L^p(\R^d)}
\lesssim_{p,N,\alpha} R^{-\alpha}
$$
for all $R \ge 1$.
\end{lemma}

\begin{proof}  In view of \eqref{warp} and the triangle inequality, it suffices to show that
$$ \sum_{\vec r \in (\Z^d)^k: r_{c(m)} = 0: R \leq \langle \vec r \rangle < 2R} \prod_{i=1}^{m-1} K_0( r_{c(i+1)} - r_{c(i)} )
\lesssim_\alpha R^{-\alpha}$$
for $\alpha$ sufficiently small.  Now recall the kernels $K_\alpha$ defined in \eqref{K-def}. From the triangle inequality (and the surjectivity of $c$) we see that
$$ \prod_{i=1}^{m-1} K_0( r_{c(i+1)} - r_{c(i)} ) \lesssim_\alpha R^{\alpha} \prod_{i=1}^{m-1} K_\alpha( r_{c(i+1)} - r_{c(i)} )$$
whenever $R \leq \langle \vec r \rangle$.  Thus it will suffice to show that
\begin{equation}\label{split_warp}
S_\alpha(c)\coloneqq  \sum_{\vec r \in (\Z^d)^k: r_{c(m)} = 0} \prod_{i=1}^{m-1} K_\alpha( r_{c(i+1)} - r_{c(i)} ) \lesssim_\alpha 1
 \end{equation}
 for all $\alpha\leq \alpha_0(m)>0$.
 
 In order to prove this we need a simple lemma on colorings. For that end we
 need some terminology. Let 
 $c\colon \{ 1,\ldots, m\}\to \{ 1,\ldots k\}$ be a (surjective) coloring. 
 Fix $k'\in \{ 1,\ldots k\}$, and denote $\ell=\#c^{-1}(k').$  One defines
 in an obvious way the
 coloring $c'\colon \{ 1,\ldots, m-\ell\}\to \{ 1,\ldots k-1\}$ that is obtained by
 removing  color $k'$  from $c$. More precisely, if $c$ is thought as a
 sequence of length $m$ containing integers from $\{ 1,\ldots k\}$, the
 sequence $c'$ is obtained by taking of all occurrences $k$ from $c$, keeping the order
 of the remaining elements, and replacing
 every $j>k'$ by $j-1.$ 
 \begin{lemma}\label{color} Let $c$ be a
 non-split coloring with at least 3 colors. Then we may remove from $c$ a color $($different from $
 c(m))$ so that the remaining coloring is also non-split.
 \end{lemma}
 \begin{proof}
 We begin by  defining the convex support of a color $k'$
 as the interval $\{ j,j+1,\ldots ,j'\},$ where $j= \min\{ i\in\{1,\ldots ,k\} : c(i)=k'\}$ and $j'= \max\{ i\in\{1,\ldots ,k\} : c(i)=k'\}$.
To prove the Lemma, note first that in  case $c(1)=c(m)$ we may remove any other color and what remains is non-split. In case $c(m)\not= c(1)$ we first try to remove the color $c(1)$. If the outcome is non-split  we are done. If the outcome is split it means that
 there must be a color $k'$ whose convex  support is contained
 in the convex support of $c(1)$, especially that color is different from $c(m).$ When color $k'$ is removed it is clear that remaining coloring is non-split.
\end{proof} 
 
   We return to the proof of \eqref{split_warp} and induct on the number of colours in $c$. If there is only one color the statement is obviously true. Assume then that $c$ contains $k$ different colors
 with $k\geq 2$ and the statement is true if the number of colors does nor exceed $k-1.$ 
 Now, if $k\geq3$, according to the previous lemma there is a color $k'$ that can be removed from $c$ so that the remaining  coloring $c'$ is non-split. If $k=2$
 we just pick $k'$ to be the color different from $c(m).$ Then,  since $c$ is non-split, we may  pick  $ 1\leq j<j'\leq k$  so that $j\leq j'-2$ and  $c(i)=k'$ for all $i$ with $j'<i<j'$, but $c(i)\not= k'$ for $i=j,j'.$ We obtain
\begin{align*}
 \sum_{r_{k'}\in\Z^d}\prod_{i=j}^{j'-1}K_\alpha (r_{c(i+1)}-r_{c(i)})
 \;=\; &\sum_{m\in\Z^d}K_\alpha (m-r_{c(j)})K_\alpha (r_{c(j')}-m)\\
 \lesssim_\alpha\; &K_{2\alpha}( 
 r_{c(j)}-r_{c(j')}).
\end{align*}
We thus obtain
 $$
 S_\alpha (c)\leq S_{3\alpha}(c')\lesssim 1,
 $$
 and by induction the claim follows if we take (say) $\alpha\leq \alpha_0\coloneqq 3^{-(m+1)}$
 initially.

\end{proof}

Now we can finally show that $T_{c,\delta}1$ is a multi-scale function.  Fix $1 < p < \infty$, let $N > d$ be large, and let $\eps_0 > 0$ be a small number to be chosen later.  Let us first consider the ``non-local'' contribution when $\langle \vec r \rangle \geq R \coloneqq  \delta^{-\eps_0}$.  From 
Lemma \ref{discret} (applied with $p$ close to infinity) we see that
$$ \|[f_i]_\delta\|_{l^\infty(\Z^d)} \lesssim_\eps \delta^{-\eps}$$
for all $\eps > 0$.  From Lemma \ref{rcm} and the triangle inequality we thus see that
$$ \| \langle \cdot \rangle^N \sum_{\underset{|\vec r|}{\vec r \in (\Z^d)^k: r_{c(m)} = 0}}
 (\prod_{i=1}^m [f_i]_\delta(n + r_{c(i)})) G_{\vec r} \|_{L^p(\R^d)} \lesssim_{p,N,\eps} \delta^{-\eps} R^{-\alpha} 
 |[f_m]_\delta(n)|$$
for some $\alpha > 0,$ assuming $|\vec r|\geq R,$ and so
$$ \| \langle \cdot \rangle_{[n,\delta]}^N \sum_{\vec r \in (\Z^d)^k: r_{c(m)} = 0}
 (\prod_{i=1}^m [f_i]_\delta(n + r_{c(i)})) (G_{\vec r})_{[n,\delta]} \|_{L^p(\R^d)} \lesssim_{p,N,\eps} \delta^{-\eps} \delta^{d/p} R^{-\alpha} |[f_m]_\delta(n)|$$
 for all $n \in \Z^d$.  Taking $l^p(\Z^d)$ norms of both sides and using H\"older and Lemma \ref{discret} we obtain (if $N$ is large enough)
$$ \| \sum_{n \in \Z^d} \sum_{\vec r \in (\Z^d)^k: r_{c(m)} = 0}
 (\prod_{i=1}^m [f_i]_\delta(n + r_{c(i)})) (G_{\vec r})_{[n,\delta]} \|_{L^p(\R^d)} \lesssim_{p,N,\eps} \delta^{-\eps} R^{-\alpha}$$
which is negligble by the choice of $R$ if we let $\eps$ be sufficiently small.  Thus we only need to consider the ``local'' contribution when $\langle \vec r\rangle < R$.  We split this local contribution into three pieces: the main term
\begin{equation}\label{zero}
\sum_{n \in \Z^d} \sum_{\vec r \in (\Z^d)^k: r_{c(m)} = 0; \langle \vec r \rangle < R}
 [\prod_{i=1}^m f_i]_{\delta}(n) (G_{\vec r})_{[n,\delta]}(x),
\end{equation}
a first error term
\begin{equation}\label{first}
\sum_{n \in \Z^d} \sum_{\vec r \in (\Z^d)^k: r_{c(m)} = 0; \langle \vec r \rangle < R}
 \left( \prod_{i=1}^m [f_i]_\delta(n) - [\prod_{i=1}^m f_i]_{\delta}(n) \right)  
  (G_{\vec r})_{[n,\delta]}(x)
\end{equation}
and a second error term
\begin{equation}\label{second}
\sum_{n \in \Z^d} \sum_{\vec r \in (\Z^d)^k: r_{c(m)} = 0; \langle \vec r \rangle < R}
 [ \prod_{i=1}^m [f_i]_\delta(n+r_{c(i)}) - \prod_{i=1}^m [f_i]_\delta(n) ]  
  (G_{\vec r})_{[n,\delta]}(x).
\end{equation}
Let us first consider the main term \eqref{zero}.  By Lemma \ref{envprod}, $\prod_{i=1}^m f_i$ is an envelope function.  From Lemma \ref{rcm} we see that the function 
$$\sum_{\vec r \in (\Z^d)^k: r_{c(m)} = 0; \langle \vec r \rangle < R} G_{\vec r}$$
is a localized function.  By Definition \ref{mtp}, we thus see that \eqref{zero} is a multiscale tensor product of an envelope function and a localized function, and is thus  a multiscale function.

To conclude the proof of Proposition \ref{mainprop}, and hence Theorem \ref{mainthm}, it suffices to show that the expressions \eqref{first} and \eqref{second} are negligible.  For this we shall just use  \eqref{warp} rather than the more sophisticated estimate in Lemma \ref{rcm} (in particular, we do not need the non-split hypothesis).

Now we turn to \eqref{first}.  Let $1 < p < \infty$, and pick any $N > d$.  Using the triangle inequality, followed by Lemma \ref{loc}, we can estimate the $L^p(\R^d)$ norm of \eqref{first} by
$$ \lesssim_{p,N} \sum_{\vec r \in (\Z^d)^k: r_{c(m)} = 0; \langle \vec r \rangle < R} 
( \sum_{n \in \Z^d} ( |\prod_{i=1}^m [f_i]_\delta(n) - [\prod_{i=1}^m f_i]_{\delta}(n)| \| \langle \cdot \rangle_{[n,\delta]}^N (G_{\vec r})_{[n,\delta]}(x) \|_{L^p(\R^d)} )^p )^{1/p}.$$
Applying a rescaled version of \eqref{warp}, we can estimate this by
$$ \lesssim_{p,N} \delta^{d/p} \sum_{\vec r \in (\Z^d)^k: r_{c(m)} = 0; \langle \vec r \rangle < R} \prod_{i=1}^{m-1} K_0( r_{c(i+1)} - r_{c(i)} ) 
\| \prod_{i=1}^m [f_i]_\delta(\cdot) - [\prod_{i=1}^m f_i]_{\delta}(\cdot)\|_{\ell^p(\Z^d)}.$$
Observe that on the ball of radius $R$, $K_0$ has an $\ell^1$ norm of $O_\eps(\delta^{-\eps})$ for any $\eps$.  Thus we can estimate the previous expression by
$$ \lesssim_{p,N,\eps} \delta^{d/p-\eps}
\| \prod_{i=1}^m [f_i]_\delta(\cdot) - [\prod_{i=1}^m f_i]_{\delta}(\cdot)\|_{\ell^p(\Z^d)}$$
for any $\eps > 0$.  Applying Lemma \ref{commutator} and H\"older's inequality repeatedly, we can thus estimate this expression by
$$ \lesssim_{p,N,\eps} \delta^{\eps_p - \eps}$$
for some $\eps_p > 0$ depending on $p$.  Setting $\eps \coloneqq  \eps_p/2$ (say) we see that \eqref{first} is negligible as desired.

Finally, we estimate \eqref{second}.  Again let $1 < p < \infty$, and pick any $N > d$.  Arguing as before, especially using the $\ell^1$ norms on $K_0$ on ball of radius $R$ we can estimate the $L^p(\R^d)$ norm of \eqref{second} by
$$  \lesssim_{p,N,\eps} \delta^{d/p-\eps}
\| \prod_{i=1}^m [f_i]_\delta(\cdot+r_{c(i)}) - \prod_{i=1}^m [f_i]_\delta(\cdot)\|_{\ell^p(\Z^d)}.$$
Using the crude estimate
$$ \left|\prod_{i=1}^m a_i - \prod_{i=1}^m b_i \right|\lesssim \sum_{i=1}^m |a_i - b_i| \prod_{j \neq i} (|a_i| + |b_i|),$$
the triangle inequality, and the already familiar estimate  
$$
\|[f_i]_\delta(\cdot+r_{c(i)})-[f_i]_\delta(\cdot)\|_{\ell^q(\Z^d)}\lesssim \delta^{-d/q+\varepsilon_q}
$$
  we get by H\"older that the $L^p(\R^d)$-norm of \eqref{second}
  has the upper bound
$\lesssim_{p,N,\eps} (R\delta)^{\eps_{mp}} \delta^{-\eps}$, where
$\eps_{mp} >0$.  By choice of $R$ we see  by choosing $\eps$ sufficiently small that \eqref{second} is negligible as required.  
This proves Proposition \ref{mainprop}.

The only thing that remains to be done to complete the proof of Theorem \ref{mainthm} is to get rid of the assumption that the envelope functions are compactly supported.
Recall \eqref{muph} and denote in the general case
$Z_\delta\coloneqq \int_{\R^d} \mu_{m,\delta}(x,\omega) {\phi}(x)\ dx$
and for $R>0$ set $Z_{\delta,R}\coloneqq \int_{\R^d} \mu_{R,m,\delta}(x,\omega) {\phi}(x)\ dx,$ where $\mu_{R,m,\delta}$ is obtained
from $\mu_{m,\delta}$ by replacing each envelope function
$f_j$ in its definition by $f_j 1_{B(0,R)}.$ Then for a suitably chosen sequence $R_k\uparrow\infty$ we have
$\|Z_{\delta,R_k}- Z_{\delta}\|_{L^2(\R^d\times\widetilde\Omega)}\leq 2^{-k}$ as $k\to\infty$, according to
\eqref{stochintegralholder}, and combined with H\"older's inequality this easily implies that $Z_{\delta,R_k}\to Z_{\delta}$ almost surely as $k\to\infty.$ We know that there are complex
numbers $z_k$ and $c,\varepsilon >0$ so that
\begin{equation}\label{zdk}
\P (|Z_{\delta,R_k}-z_k|>\delta^\eps)\leq c\delta^\varepsilon,
\end{equation}
and the argument in the present section verifies that
$c$ is independent of $k\geq 1.$ As $\E |Z_{\delta,R_k}|^2$ is uniformly bounded in $\delta$ and $k$, we deduce that the sequence
$(z_k)$ is uniformly bounded, and by moving to a subsequence we may assume that $z_k\to z$ as $k\to\infty.$
One obtains the desired inequality simply by letting $k\to\infty$
in \eqref{zdk}. The proof is complete.

\section{Quasiconformal homogenization}\label{se:qchomogenization}

Our next task is to apply Theorem \ref{mainthm} with Corollary \ref{co:main} to homogenization of quasiconformal maps. Here it turns out convenient  to proceed via the principal solutions, c.f., Subsection \ref{ss:bird}. This, on the other hand, requires us to first  make use the Theorems in the setting of compactly supported envelope functions. Once that is done the application to general quasiconformal homogenization poses no substantial difficulties. However, for the  reader's convenience  we present rather complete details.

We refer to e.g., \cite[Section 1]{CG} for a quick account of basic facts 
 about planar quasiconformal maps, and to \cite{AIM} for a comprehensive exposition on the topic. 
  Throughout this section $T$ stands for the Beurling operator \eqref{eq:beurl}. Recall from the introduction that a (quasiconformal) complex dilatation $\mu$ is a complex valued measurable function on the plane whose sup-norm is strictly less than 1, that a 3-point normalized homeomorphism  of the extended plane $f\colon \overline{\C}\to\overline{\C}$ fixes points $0,1$ and $\infty$, and that the measurable Riemann mapping theorem quarantees existence and uniqueness of a 3-point normalized homeomorphic  $W^{1,2}_{loc}$-solution to the Beltrami equation $\deeb f=\mu \dee f$ for any quasiconformal dilatation. 

In preparation for the  proof ot Theorem \ref{th:main}(i), we begin with a few simple deterministic lemmas,  which are modifications of  well-known methods in the  theory of planar quasiconformal mappings. 
Our first lemma 
shows that weak convergence of each 
individual term in the Neumann series is enough to guarantee  uniform convergence of the corresponding 
principal solutions and locally uniform convergence of the 3-point normalized solutions.

\begin{lemma}\label{lisaco:2.1} 
Let  us assume that for any $j\geq 1$ the dilatation  $\mu_j$ satisfies $ \|\mu_j\|_\infty\leq k<1$ and $\supp(\mu_j)\subset B$, where $B\subset\C$  is a ball. 
Denote the $m$-th term in the Neumann series  for $\mu_j$ by  
$$
\psi_{m,j}\coloneqq  \mu_jT\mu_j\ldots T\mu_j,
$$
where $\mu_j$
appears $m$ times, $m\geq 1$.
Assume also that for every fixed  $m$ there is the weak convergence in $L^p(\C)$ 
$$
\psi_{m,j}\overset{w}{\to} \psi_m \quad \textrm{as}\;\; j\to\infty,
$$
for all $1 < p< \infty.$ Then the solution
$F_j$ of the Beltrami equation $\deeb F_j=\mu_j\dee F_j$,
normalized by the 3-point condition, converges locally
uniformly to a  $k$-quasiconformal  limit $F_\infty\colon \C\to\C$.
\end{lemma}

\begin{proof}  Let first $f_j$ be the principal solution that has the representation
$$
f_j=z+\sum_{m=1}^\infty  C\psi_{m,j},
$$
where $C$ is the Cauchy transform.
 All the functions $\psi_{m,j}$
are supported in the ball $B$, and by the standard properties of $T$ (see \cite[Section 4.5.1]{AIM}),
we have $\|\psi_{m,j}\|_{L^p(B)}\leq c a^m$ for all $j$, where
$a=a(p, k)<1$  as soon as if we fix $p>2$ close enough to $2$.

It is well-known that for $p>2$ the map $C\colon L^p(B)\to C^\alpha(\C)$ is bounded and compact e.g., 
\cite[Thms 4.3.11 and 4.3.14]{AIM} 
for $\alpha\in (0, 1-2/p)$. Here clearly the homogeneous norm for $C^\alpha$ used in \cite{AIM} can be replaced by the non-homogenous norm 
$$
\|f\|_{C^\alpha}(\C)\coloneqq \|f\|_{L^\infty(\C)}+\sup_{z,w}|f(z)-f(w)||z-w|^{-\alpha}
$$
 by the good decay of the Cauchy transforms of compactly supported functions. 
We may thus deduce from the weak convergence of $\psi_{m,j}$ in $L^p(B)$ that for each $m\geq 1$ the term $C\psi_{m,j}$
converges in the $C^\alpha (\C)$-norm  to an element $g_m\in C^\alpha(\C)$. Moreover, we have the uniform bounds  $\|C\psi_{m,j}\|_{C^\alpha(\C)}\leq Ca^m$ and $ \| g_m\|_{C^\alpha (\C)}\leq Ca^m$ for all $m,j\geq 1$. This clearly yields the uniform convergence
of the principal solutions  
\begin{equation}\label{psl}
f_j\to f_\infty=z+\sum_{m=1}^\infty  C\psi_{m} \quad\text{as}\quad j\to\infty .
\end{equation}
The  limit $f_\infty$ is $k$-quasiconformal from the normal family property of hydrodynamically normalized
$k$-quasiconformal maps with dilatations supported in a fixed ball.

Finally, to treat the $3$-point normalized solutions $F_j$, simply
observe we may write them in  terms of the  principal solution as
$$
F_j(z)=(f_j(1)-f_j(0))^{-1}(f_j(z)-f_j(0)).
$$
Thus $(F_j)$ converges uniformly to the $k$-quasiconformal map 
$$ 
F_\infty (z)\coloneqq (f_\infty(1)-f_\infty(0))^{-1}(f_\infty(z)-f_\infty(0)).
$$
\end{proof}

Our second  auxiliary result verifies that normalized $k$-quasiconformal maps whose dilatations agree in a large
ball are close to each other near the center of the ball.
\begin{lemma}\label{qc_locality} Let $k<1$ and  assume that both $f\colon \C\to\C$ and $g\colon \C\to\C$  are $k$-quasiconformal homeomorphisms that satisfy the  $3$-point normalization and,  moreover
$$
\mu_g=\mu_f \quad\textrm{in}\;\; B(0,L),
$$
where $L\geq 1.$ Then for any $R<L$ we have
$$
\sup_{|z|\leq R}|g(z)-f(z)|\leq \varepsilon (L,k,R),
$$
where  $\lim_{L\to\infty}\varepsilon (L,k,R)= 0$ for any fixed $k,R$.
\end{lemma}

\begin{proof} First of all, quasisymmetry (see \cite[Def. 3.2.1 and Thm 3.5.3]{AIM}) and the normalization of $g$ imply that
$g(B(0,R))\subset B(0,r_1)$ and $g(B(0,L))\supset B(0,r_2)$
with $r_{1}=r_1(R,k)$   and $r_2=r_2(L,k)\to\infty$ as $L\to\infty.$
Writing $f=h\circ g$, it follows that  $h$ is analytic in $B(0,r_2)$ with $h(0)=0$  and $h(1)=1$.
Then the function 
$$
H(z)\coloneqq  r_2^{-1}h(r_2z)
$$
is analytic and univalent in $B(0,1)$ and satisfies  the normalization $H(0)=0$, $H(1/r_2)=1/r_2$. By 
the Koebe type estimates (\cite[(2.74)]{AIM}) it is clear that $H'(0)\to 1$ as $L\to\infty$. Since the second derivative of $H$ has a universal bound on say $B(0,1/2)$ (\cite[Thm. 1.8]{CG}) we deduce that 
for any given $\varepsilon >0$ we  have for large enough $L$
$$
|H(z)-z|\leq \varepsilon |z| \leq \varepsilon r_1/r_2\quad \textrm{for} \quad |z|<r_1/r_2.
$$
This implies that $|f(z)-g(z)|<\varepsilon r_1$ for $|z|<R,$ proving the lemma.
\end{proof}

Next we have a global variant of Lemma \ref{lisaco:2.1}.

\begin{lemma}\label{lisaco:2.2}
Let the dilatations  $\mu_j$ satisfy $|\mu_j|\leq k<1$ for   $j=1,2,\ldots$. For any $L>1$ we write
$\mu_{j,L}\coloneqq \mu_j 1_{B(0,L)}$ and set
$\psi_{m,j,L}\coloneqq  \mu_{j,L}T\mu_{j,L}\ldots T\mu_{j,L},$ where $\mu_{j,L}$ appears $m$ times.  
Assume  that for every $m\geq 1$ and $L>1$ there is the weak convergence
$$
\psi_{m,j,L}\overset{w}{\to} \psi_{m,L} \quad \textrm{as}\;\; j\to\infty
$$
in $L^p(\C)$ for all  $1 < p< \infty.$ Then the $3$-point normalized solution
 $F_j$ of the Beltrami equation $\deeb F_j=\mu_j\dee F_j$
 converges locally uniformly on $\C$ to a
$k$-quasiconformal homeomorphism  $F$.
\end{lemma}

\begin{proof} Fix $R>0$. For any $L=1,2,3,\ldots$  let  $F_{j,L}$ be the $3$-point-normalized solution to the Beltrami equation
$$
\deeb F_{j,L}=\mu_{j,L}\dee F_{j,L} .
$$
By Lemma \ref{lisaco:2.1}, for every $L\geq 1$ we have uniform convergence $F_{j,L}\to F_{\infty,L}$ as $j\to\infty$, where
$F_{\infty,L}$ is a $k$-quasiconformal homeomorphism. 
Given $\varepsilon >0$, Lemma \ref{qc_locality} shows that we may choose $L_0\coloneqq L_0(k,\varepsilon, R)$ so  that 
$$
 |F_{j,L}-F_{j,L'}|\leq \varepsilon \quad
\textrm{in}\;\; z\in B(0,R),\qquad \textrm{for}\quad L,L'\geq L_0 .
$$
A fortiori,
$$
 |F_{\infty,L}-F_{\infty,L'}|\leq \varepsilon\quad
\textrm{in}\;\; z\in B(0,R),\qquad \textrm{for}\quad L,L'\geq L_0 .
$$
We  deduce that the sequence 
$(F_{\infty,L})_{L\geq 1}$ is Cauchy in $C(B(0,R)), $ so that $F_{\infty,L}\to F_\infty$
uniformly on $B(0,R)$. Since $R$ was arbitrary, we see that $F_\infty$ is a 3-point normalized $k$-quasiconformal homeomorphism of the plane. 

It remains to check  that also $F_j\to F_\infty$
uniformly on $B(0,R)$ for any given $R\geq 1$. To  this end, take $L\geq L_0$ and estimate
\begin{align*}
&\limsup_{j\to\infty}\|F_{j}-F_\infty\|_{C(B(0,R))}\\
\leq& 
\limsup_{j\to\infty} \| F_{j}-F_{j,L}\|_{C(B(0,R))} + \limsup_{j\to\infty}\| F_{j,L}-F_{\infty,L}\|_{C(B(0,R))} + \| F_{\infty,L}-F_\infty\|_{C(B(0,R))}\\
&\leq \varepsilon +0+\varepsilon \;=\; 2\varepsilon,
\end{align*}
where  we  used Lemma \ref{qc_locality} again to estimate the first term.
\end{proof}

We are ready to establish the first statement in Theorem \ref{th:main}(i).

\begin{proof}[Proof of Theorem \ref{th:main}(i)] 
Let us first assume that  the Beltrami envelope function $\phi$ in the statement of Theorem \ref{th:main}(i)   (see Definition \ref{de:ref})
is compactly supported.
Observe that in this case $\phi$ is an
envelope function in the sense of Section \ref{dmd-sec} (Definition \ref{def:env}), since
taking $R$ large enough in  Definition \ref{de:ref} we may
apply the bound $|\phi|\leq 1$ to
obtain for any $1 < p< \infty.$ 

$$
\|\Delta_h\phi\|_{L^p(\C)}\leq 2^{1-1/p} |supp(\phi)|^{1/p}
\|\Delta_h\phi\|_{L^1(\C)}^{1/p}\leq C'|h|^{\alpha/p}\quad
\textrm{for}\;\; |h|\leq 1.
$$

Lemma \ref{prod3} shows that $\phi\, B_\delta$ is
a stochastic multiscale function. By Corollary \ref{co:main}, for each $m\geq1$
there exists a (deterministic) limit function $\psi_m$ such that, with probability one,
$\psi_{m,j}\coloneqq \mu_{2^{-j}}T\mu_{2^{-j}}\ldots T\mu_{2^{-j}}$ converges weakly
to  $\psi_m$  in $L^p(\C)$ for each $1 < p< \infty$ and each $m$. The statement of part (i) then follows from Lemma 
\ref{lisaco:2.1}.

In the case where the envelope $\phi$ is not compactly supported,
we use Lemma \ref{lisaco:2.2}  to reduce to the compactly supported case. For this reduction it is enough to note that  $\phi 1_{B(0,R)}$ is an envelope function if $\phi$ is a Beltrami envelope function, by essentially the same argument as above -- one uses additionally the observation that a characteristic function of a ball is an  envelope function.
\end{proof}

\begin{lemma}\label{le:locality} Assume that $k\in [0,1)$ and let $(f_j)$ and $(g_j)$ be sequences of locally uniformly convergent  $k$-quasiconformal maps in a
 domain $\Omega\subset\C$ such that the limit functions $f=\lim_{j\to\infty} f_j$ and $g=\lim_{j\to\infty} g_j$ are  non-constant.
Assume also that $  |\mu_{f_j}-\mu_{g_j}|\leq \varepsilon$ in $\Omega$ for all  $j\geq 1.$
Then 
$$
|\mu_f-\mu_g|\leq \varepsilon\frac{1+k^2}{1-k^2}\quad \textrm{in}\;\; \Omega .
$$
\end{lemma}

\begin{proof} Take any ball  $B(z_0,R)\subset\Omega$. By  considering
$f_j(z)-f_j(z_0)$  and $g_j(z)-g_j(z_0)$ instead, we may assume that $g_j(z_0)=f_j(z_0)=0$ for all $j.$ The assumptions together with the quasisymmetry property of the maps imply that
if  $r>0$ is taken small enough, then $B(0,r)\subset g_j(B(z_0,R))$
for all $j\geq j_0$, and hence the map $f_j\circ g_j^{-1}$
is well-defined  in $B(0,r)$
for $j\geq j_0$. We may compute (see \cite[(13.37)]{AIM})
\begin{equation}\label{mu_comp}
\mu_{ f_j\circ g_j^{-1}}(w)=\left({ \frac{\mu_{f_j}-{\mu_{g_j}}}{1-\mu_{f_j}\overline{\mu_{g_j}}}\frac{\dee g_j}{\overline{\dee g_j}} }\right)\circ g^{-1}(w),\quad \textrm{for a.e.}\;\; w\in B(0,r).
\end{equation}
In particular, $|\mu_{  f_j\circ g_j^{-1}}|\leq \varepsilon(1-k^2)^{-1}$ and letting
$k\to\infty$ we infer by the local uniform convergence that
$|\mu_{f\circ g^{-1}}|\leq \varepsilon(1-k^2)^{-1}$ in the neighbourhood of $z_0.$ In particular, applying  formula \eqref{mu_comp} to $f$ and $g$ we obtain
$$
\left|\mu_{f}-\mu_{g}\right|\leq (1+k^2)\left|\frac{\mu_{f}-\mu_{g}}{1-\mu_{f}\overline{\mu_{g}}}\right|\leq 
(1+k^2)|\mu_{f\circ g^{-1}}|\leq \varepsilon\frac{1+k^2}{1-k^2}.
$$
\end{proof}

Our next auxiliary result is quite specialized to our situation. Note that the existence of the deterministic homogenization limit $F_\infty$ is guaranteed by part 
(i) of Theorem \ref{th:main} that we already verified.

\begin{lemma}\label{le:constant} 
Suppose  in Theorem \ref{th:main}(i) the Beltrami envelope function $\phi$  is constant on the complex plane.  Then
the dilatation $\mu$ of the homogenization limit  $F_\infty\colon \C\to\C $ is constant on  $\C$, 
and therefore $F_\infty$ is linear: \; 
$$
F_\infty(z)=\frac{1}{1+A}z+\frac{A}{1+A}\overline{z},
$$
 where the constant $A=\mu_{F_\infty}$ satisfies $|A|<1$.
\end{lemma}
\begin{proof}
\newcommand{\Q}{{\mathbf Q}}
Let $F_j$ be defined via \eqref{eq:Fj}, and let  $B_{2^{-j}}$ be the random bump field defined by \eqref{eq:rbf1}.
Denote by $\Q^2_d$ the set of dyadic rational points in $\C$, i.e., numbers
of the form $(n+mi)2^{-\ell},$ where $m,n$ and $\ell\geq 1$ are integers. Since now $\mu_{F_j}=a B_{2^{-j}}$, where $a$ is a constant with $|a|<1$, we have  for  any $b\in \Q^2_d$ 
$$
\mu_{F_j(\cdot+b)}\sim\mu_{F_j(\cdot)}\quad\textrm{for}\quad j\geq j_0(b)
$$
where $\sim$ stands for  equivalence in distribution. As a consequence of the 3-point normalization we may  write for  $j\geq j_0(b)$
$$
F_j(z)\sim a_jF_j(z+b) +c_j,
$$
where $a_j=(F_j(b+1))-F_j(b))^{-1}$ and $c_j=-a_jF_j(b).$ In the limit $j\to\infty$ we thus obtain
$$
F_\infty(z) = aF_\infty(z+b) +c
$$
with constants $a\not=0$ and $c$ that depend only on $b$. This implies that
$$
\mu_{F_\infty}(z)= \mu_{F_\infty}(z+b),
$$
where the
equality is in the sence of $L^\infty$-functions.

Therefore $\mu$ is periodic on $\C$ with 
dyadic rational periods, and this easily implies  that $\mu$
is constant. Finally, for any $A\in\D$ the linear map $z\mapsto \frac{1}{1+A}z+\frac{A}{1+A}\overline{z}$ satisfies the 3-point normalization and has dilatation $A$, whence it is the unique quasiconformal homeomorphism $\C\to\C$ with these properties.
\end{proof}

We are now ready to prove the second statement of Theorem \ref{th:main}.

\begin{proof}[Proof of Theorem \ref{th:main}(ii)] 
We first define the function $h_{(g,X)}$
with the help of a reference homogenization limit. For any $a\in\{ |w|<1\}$ let $F_a$
be the unique deterministic limit map of the homogenization problem
$$
\deeb F_{a,j}(z)=a B_{2^{-j}}(z)\dee F_{a,j}.
$$
By Lemma \ref{le:constant} $F_a$ has constant dilatation in the whole plane; let us denote by $h_{(g,X)}(a)$  its value. Part (i) of
Theorem \ref{th:main} and Lemma \ref{le:locality} yield immediately that the map $a\to h_{(g,X)}(a)$
 is continuous. 

Assume next that the envelope function $\phi$ is continuous in a neighbourhood of $z_0$.
with $\phi(z_0)=a$. Then the dilatations of the sequences $F_{a,j}$
and $F_j$ (where $F_j$ is as in the Theorem, see \eqref{eq:Fj}) are $\varepsilon$-close
in a small enough neighbourhood $U$ of $z_0$. 
Thus Lemma \ref{le:locality} shows that the dilatation of
the homogenization limit $F_\infty$ differs from $h_{(g,X)}(a)$ by less than
$\varepsilon(1+k^2)(1-k^2)^{-1}$ in a small enough neighbourhood $U$,
and we deduce the continuity of $\mu_{F_\infty}$ and the equality
$ \mu_{F_\infty}(z_0) =h_{(g,X)}(a)=h_{(g,X)}(\phi(z_0)).$
\end{proof}

We state one more auxilary result which actually contains a more general statement than what is needed in the last part of Theorem \ref{th:main}.
  
\begin{lemma}\label{le:identitylimit}
Assume that $g$ is invariant under rotation by the angle $\pi/2 $: 
$$
g(z,t)=g(iz,t)\quad \textrm{for all }\quad z\in\C, t\in\R.
$$
Moreover, assume that $X$ is such that the random field
$g(\cdot ,X)$ is symmetric, i.e $g(\cdot ,X) \sim -g(\cdot ,X)$.
Then $h_{(g,X)}(a)=0$ for every $a\in\D.$
\end{lemma}
\begin{proof}
Let $B_{\delta}$ be the random bump field defined by \eqref{eq:rbf1}.
The symmetry of $g$ together with the indepence of the $X_n$ implies the symmetry of $B_{\delta}$.
Fix $a\in\D.$
For $j\geq 1$, let $F_j$ solve the random Beltrami equation 
\begin{equation}\label{eq:symmetricbeltrami}
\deeb F_j=a B_{2^{-j}}(z) \dee F_k,
\end{equation}
and denote  $\widetilde F_j(z)=(F_j(i))^{-1}F_j(iz)$. One computes that $\mu_{\widetilde F_j}(z)=-\mu_{F_j}(iz)$. The assumptions of the lemma thus verify that 
$$
\mu_{\widetilde F_j}\sim \mu_{ F_j},
$$
whence in the limit $j\to\infty$ we deduce that $F_\infty (z)= cF_\infty (iz)$ with a constant $c\not=0$. By lemma \ref{le:constant}  we obtain the identity
$$
\frac{1}{1+A}z+\frac{A}{1+A}\overline{z}= c\Big(\frac{1}{1+A}iz-\frac{Ai}{1+A}\overline{z}\Big) \qquad \textrm{ for all}\;\;z\in\C.
$$
 The above identity is possible only if $c=-i$ and $A=0$. Thus $h_{g,X}(a)=A=0$ as was to be shown.
\end{proof}

\begin{proof}[Proof of Theorem \ref{th:main}(iii)] 
The statement that
for both of the models \eqref{mu-ex2} and \eqref{mu-ex3,5}
  the deterministic limit map is the identity map follows
immediately  from Lemma \ref{le:identitylimit}
and Theorem \ref{th:main}(ii). 
 
Finally, we show that in the generic case, the limit map is not
the identity or equivalently, that the Beltrami coefficient of the limiting map is not zero. 
To this end, we consider a very simple case of the general model. Fix a bump function $g\in C^\infty_0((0,1)^2)$  with $||g||_{\infty}\leq1$ and consider the sequence of
random dilatations $\mu_{j,a}$ that depend on the complex
parameter $a\in\D$
$$
\mu_{j,a}(z)=a 1_{[0,1]^2}(z)\sum_{n\in\Z^2}\varepsilon_n g(2^jz-n),
$$
where the $\varepsilon_n$ are an independent sequence of random
signs $\pm 1.$ Let $f_{j,a}$ be the principal solution of
the corresponding Beltrami equation, and denote by $f_a$
the almost sure deterministic limit function $f_a=\lim_{j\to\infty}
f_{j,a}.$ 
Using notation as in Lemma \ref{lisaco:2.1} (with $\mu_j=\mu_{j,a}$)
we see from \eqref{psl} that $f_a$
has the (power series) representation
$$
f_a(z)=z+\sum_{m=1}^\infty (C \psi_m)(z)=z+\sum_{m=1}^\infty a^m(C\widetilde \psi_m)(z),
$$
with $\widetilde \psi_m=\lim_{j\to\infty}\widetilde\psi_{m,j},$ where $\widetilde \psi_{m,j}\coloneqq \mu_{j,1}T\mu_{j,1}\ldots T\mu_{j,1}$ and where the almost sure weak convergence to the (deterministic) limit  $\widetilde \psi_m$ in $L^p(\C)$ for each $p>1$ again follows from  Corollary \ref{co:main}.
We claim that $f_a$ is non-linear (equivalently, the 3-point normalized limit is not the identity)
for all but countably many values of $a\in\D$, unless
$\widetilde \psi_m$ is identically $0$ for all $m:$ To see this, notice that $f_a(z)-z\to0$ as $z\to\infty,$ 
so that $f_a$ cannot be linear unless $f_a(z)-z$ is independent of $z$.
By interpreting $(C\widetilde \psi_m)(z)$ as the Taylor coefficients in the  power series representation of $a\mapsto f_a(z)-z$ above, we see that $f_a$ is non-linear for all but countably many values of $a$ unless
$C\widetilde \psi_m(z)$ is independent of $z$ for all $m,$ or equivalently $\widetilde \psi_m \equiv 0$.

It thus suffices to give an example with $\widetilde \psi_2\not\equiv 0.$
Let $h\in C^\infty_0(\C)$ be a compactly supported test function that equals 1 on $[0,1]^2$.
For $j\geq 1$ set
$$
Y_j\coloneqq \int_{\C} h\widetilde\psi_{2,j}=\int_{[0,1]^2} \mu_{j,1} T\mu_{j,1}\qquad \textrm{and}\quad Y\coloneqq \int_{\C}h\widetilde\psi_2 =\int_{[0,1]^2}\widetilde\psi_2.
$$
Then almost surely $Y=\lim_{j\to\infty}Y_j$ and $Y$ is a deterministic constant. 
We note that in this special case, the convergence is not difficult to prove directly without resorting to our general theory. In any case, we claim that the limit is non-zero for a suitable choice of $g$. As the random variables $Y_j$ are uniformly bounded, we actually have $Y=\lim_{j\to\infty}\E Y_j$.    Since the supports of 
$g(2^jz-n)$ are disjoint for different values of $n$,
and  $\E \varepsilon_n \varepsilon_{n'}=\delta_{n,n'}$
we may compute
\begin{equation}\label{eq:nonzero}
\E Y_j =\sum_{n\in\Z^2:\; 2^{-j} n\in [0,1)^2}\int_\C g(2^jz-n)Tg(2^jz-n)dz =\int_\C g(z)Tg(z)dz,
\end{equation}
where in the last step we used the translation and scaling invariance of $T.$ 

It remains to verify that
$g\in C^\infty_0((0,1)^2)$ can be chosen so that the last integral in \eqref{eq:nonzero}  is not identically zero. 
The following example can be generalized to all kernels that are not odd. 
Fix any $\varphi\in C_0^\infty(\D)$ with $0\leq\varphi\leq 1$ and $\varphi\not\equiv 0.$ If $\int_\C\varphi T\varphi =0,$ then setting $\varphi_A\coloneqq \varphi(\cdot-A)+\varphi(\cdot+A)$ we have $\int_\C\varphi_AT\varphi_A  \sim 2\times \frac{-1}{\pi}(\int \varphi)^2 (2A)^{-2} \not=0$ as $A\to\infty.$
By scaling and translating the support
may be taken to be in $(0,1)^2$, and the choice $g=\varphi_A$ for large enough $A$ completes the proof
of Theorem \ref{th:main}.  
\end{proof}

\smallskip

We  next sketch 
an alternative statement of  the solution to the homogenization problem,
replacing `almost sure convergence' by `convergence in probability'. Then there is no need to restrict to subsequences
of $\delta\to 0.$ In order to rephrase
Theorem \ref{th:main} in this manner, consider  the principal solution $f_\delta$ of the homogenization problem
\begin{equation}\label{eq:problem}
\deeb F_\delta=\phi B_{\delta}  \dee F_\delta.
\end{equation}
In the case where the envelope function $\phi$ is compactly supported,
we know that the terms $\psi_{m,\delta}$ in the corresponding Neumann-series are all supported in a ball $B(0,R)$, where $R$ is independent of $\delta.$ Each term in the series converges weakly in probability in $L^p(B(0,R))$ as $\delta\to 0,$ i.e., 
for any $h\in L^{p'}$ there is the convergence in probability 
$$
\int_{\C} h\psi_{m,\delta}\to \int_{\C} h \psi_{m}.
$$
Moreover, the Neumann series converges  $L^p(\C)$, with an exponentially decaying remainder term, uniformly with respect
to $\delta>0$. All this easily implies a norm convergence in $C^\alpha$ (compare the proof of Lemma \ref{lisaco:2.1}), i.e.
$$
\P \big( \| f_\delta -f\|_{C^\alpha(\C)}>  t \big)<  t 
$$
for all $t >0$ as soon as $\delta <\delta_0( t ).$ In particular, $f_\delta\to f$ locally uniformly
in probability. Finally, we may argue exactly as  in the proof of
Theorem \ref{th:main} and dispense with the assumption that the envelope has compact support. Let us record our conclusion as a theorem:

\begin{theorem}\label{th:inprobability} Let $\mu_\delta$ be as in Theorem
\ref{th:main} and denote by $F_\delta$ the $3$-point normalized solution to the Beltrami equation \eqref{eq:problem}. Then,
$F_\delta\to F_\infty$ locally uniformly in probability as $\delta\to 0,$
where $F_\infty$ is the deterministic limit map given by Theorem \ref{th:main}. In other words, for any $R>0$ and $\varepsilon>0$
one has for $\delta<\delta_0(\varepsilon,R)$ that
$$
\P\big( \|F_\delta-F\|_{L^\infty (B(0,R))}>\varepsilon\big)<\varepsilon.
$$
\end{theorem}

As our final application to quasiconformal homogenization we consider some random mappings of finite distortion, i.e.,  homeomorphisms for which the assumption $\| \mu \|_{\infty} \leq a < 1$ is relaxed. This leads to the  study of solutions to the Beltrami equation $\partial_\zbar f = \mu \partial_z f$ where
 we only have  $| \mu(z)| < 1$ almost everywhere. From the general theory of quasiconformal mappings and mappings of finite distortion one knows  that in order to have a viable theory one needs some control on the size of the set where $| \mu(z)|$ is close to $1$. 
 For basic properties of planar maps of finite distortion 
we refer to \cite[Chapter 20]{AIM} or \cite{AGRS}.

There is a well-established theory for mappings of
G. David type, i.e., maps whose distortion function 
$$
K(z)\coloneqq \frac{1+|\mu (z)|}{1-|\mu (z)|}
$$
is exponentially integrable, namely $\exp(aK(z))\in L^1_{loc}$ for some
$a>0.$ 
With this theory in mind, a natural model for degenerate  random  Beltrami coefficients is
\begin{equation}\label{degeq}
\mu_j(z)\coloneqq \sum_{n\in\Z^2:\; 2^{-j}n\in [0,1)^2}\varepsilon_{j,n}g(2^jz-n).
\end{equation}
where $\|g\|_{L^\infty(\C)}=1,$ one has supp$(g)\subset[0,1]^2$, and for each $j\geq 1$ we
assume that $\varepsilon_{j,n}$ ($n\in\Z^2$) are complex valued
i.i.d. random variables taking values in $\D$. Their common distribution is assumed to be independent of $j$.
In this situation we have the following result:
\begin{theorem}\label{degeco:2.1}  Assume  the uniform tail estimate
\begin{equation} \label{taildecay}
 \P\left(  \frac{1+|\varepsilon_{j,n}|}{1-|\varepsilon_{j,n}|}  > t \right) \leq  e^{-  \gamma\, t}
\end{equation}
for some $ \gamma >2.$
Define the {\rm(}possibly degenerate{\rm )} Beltrami coefficients  $\mu_j$ as in \eqref{degeq}.
Then the 3-point 
normalized solutions $F_j$ of the Beltrami equation $\deeb F_j=\mu_j \dee F_j$ converge almost surely  locally uniformly to a deterministic limit homeomorphism $F:\C\to\C$. 
\end{theorem}

\begin{proof}
We start  the proof  with  a couple of auxiliary observations.
First of all, we again use that convergence of the 3-point normalizations is equivalent to convergence of the hydrodynamically normalized ones. 

Thus, we 
again consider the principal solution
\begin{equation}
\label{neumann2}
f_j(z): = z + C  \big(\sum_{m=1}^\infty \psi_{m,j}\big),
\end{equation}
of the Beltrami equation, where as before $\psi_{m,j}= \mu_j T\mu_j\ldots T\mu_j$
with $\mu_j$ occuring $m$ times. This series is  well-defined since almost surely each $\mu_j$ satisfies 
$$ \|\mu_j\|_{L^\infty (\C)}\leq \max\{|\varepsilon_{j,n}|:n\in\Z^2,\; 2^{-j}n\in [0,1)^2\} < 1.   $$

By Corollary \ref{co:main}, almost surely each of  the terms $ \psi_{m, j}$  converges  weakly to a limit $\psi_m$ in $L^p$ for every $1 < p<\infty$, and     $C( \psi_{m, j})(z)$ converges locally uniformly on $\C$.  Therefore  we expect that 
the limit map can be written again as
\begin{equation}\label{neumann3}
f_\infty= z + C  \big(\sum_{m=1}^\infty \psi_{m}\big),
\end{equation}
and in proving the convergence
one only needs to control the tail of this series. 
Our main tool will be the following 
statement:
\begin{equation}
\label{sarja}
\lim_{M\to \infty }\sup_{j\geq 1} \;  \sum_{m=M}^\infty \| \psi_{m,j} \|_{L^2(\complex)} \; = 0 \quad\quad \mbox{almost surely}\, .
\end{equation}
The proof of \eqref{sarja} is based on the following basic estimate
\cite[Theorem 3.1]{AGRS} (see also \cite{D}) 
 with $R=2$
 on the decay of the $L^2$-norm of the terms in the Neumann series.
\begin{lemma}\label{le:decay} Assume that the dilatation $\mu$ is compactly
supported, supp$(\mu)\subset B(0,R).$ If for some $p>0$
we have
\begin{equation}\label{eq:expint}
A\coloneqq \int_{B(0,R)}e^{pK(z)} dz <\infty,
\end{equation}
where $K\coloneqq \frac{1+|\mu |}{1-|\mu |},$ then for any $q\in (0,p/2)$ the $m$-th term in the Neumann-series satisfies the bound
\begin{equation}\label{eq:m_term_bound}
\|\psi_m\|_{L^2(\C)}\leq C_{R,q,A}m^{-q}
\end{equation}
\end{lemma}
Denote the distortion function of  $f_j$ by $K_j(z)\coloneqq \frac{1+|\mu_j(z) |}{1-|\mu_j(z)  |},$ where
$\mu_j$ is as in \eqref{degeq}. In view of the above lemma, 
\eqref{sarja} follows as soon as we verify that
there is $p>2$ such that
\begin{equation}\label{eq:unif_exp_bound}
 \sup_{j\geq 1}\int_{[0,1]^2}e^{pK_j(z)} dz <\infty \qquad   \textrm{almost surely}\,  .
\end{equation}
To this end, choose $q\in (1,2)$ and $p>2$ so that $pq<\gamma,$ where $\gamma>2$ is from condition \eqref{taildecay}. Denote by $Y$ a random variable  with the distribution
$$
Y\sim \exp\left(p\Big(\frac{1+|\varepsilon|}
{1-|\varepsilon|}\Big)\right)-M\qquad \textrm{with}\quad M\coloneqq \E
\exp\left( p\Big(\frac{1+|\varepsilon|}{1-|\varepsilon|}\Big)\right) ,
$$
where $\varepsilon$ has the same 
distribution as all of the variables $\varepsilon_{j,n}$.
The expectation  $M$ above is finite according to
our assumption \eqref{taildecay}, in fact
 $\E Y^q<\infty.$ The very definition of $\mu_j$ yields that
$$
\int_{[0,1]^2}e^{pK_j(z)}dz\leq M+Z_j,
$$
with
$$
Z_j\sim 2^{-2j}\sum_{\ell=1}^{2^{2j}}
Y_{j,\ell}
$$ 
where for each $j\geq 1$ the random variables $Y_{j,\ell}$ are identically distributed copies of $Y.$ In order to estimate the
tail of $Z_j$, we recall the von Bahr and Esseen estimate \cite{BE} that states  for centered i.i.d. 
random variables
$X_1,\ldots , X_N$
the inequality
$$
\E\left| X_1+\ldots X_N\right|^q\leq C_q \sum_{s=1}^N\E\left| X_s \right|^q,\qquad 1\leq q\leq 2.
$$
We obtain
\begin{align*}
\P(Z_j>1)\leq \E Z_j^q\leq 2^{-2jq}C_q2^{2j}\E Y^q=O(2^{-2(q-1)j}),
\end{align*}
and the Borel-Cantelli lemma yields that almost
surely eventually $Z_j\leq 1.$ This proves \eqref{eq:unif_exp_bound},
and we have finished the verification of \eqref{sarja}.

We will prove Theorem
\ref{degeco:2.1} using the Arzela-Ascoli theorem. To this end we need uniform modulus of continuity estimates
for both sequences $(f_j)$ and $(f_j^{-1}).$ 
Here note first that \eqref{sarja}
implies the uniform bounds (with a random constant $C$)
\begin{equation}
\|\deeb f_j\|_{L^2(\C)} =  \|\dee f_j-1\|_{L^2(\C)}\leq C,
\quad \textrm{for all}\;\; j\geq 1.
\end{equation}
Since the support of each $\mu_j$ is contained in $2\D$, this estimate together with the properties of the Cauchy transform shows that, outside $3\D,$ the functions $f_j$ are  uniformly equicontinuous and $f_j(z)-z$ is uniformly bounded.
Thus  uniform equicontinuity in all of $\C$  follows from the following useful result (see \cite{GoldsteinVodopyanov},\cite[Theorem 20.1.6]{AIM}).
\begin{lemma}[Gehring, Goldstein and Vodopyanov]\label{GV}
Assume that  $f\in W^{1,2} (4\D)$ is a homeomorphism. Then, if $z_1,z_2\in 4\D$ one has
$$
|f(z_1)-f(z_2)|^2\;\leq \;\; \frac {9\pi\int_{2\D} |\nabla f|^2}{\log (e+1/|z_1-z_2|)}.
$$
\end{lemma}

Next, the equicontinuity of the inverse maps is dealt with by another lemma (whose proof actually reduces the situation to Lemma \ref{GV}, see \cite{ISv},\cite[Lemma 20.2.3]{AIM}).
\begin{lemma}[Iwaniec and Sverak] Assume that $f$ is a (homeomorphic) principal solution of the Beltrami equation with distortion function $K$, and with $\mu$
supported in $B(0,R')$.  Then, for $z_1,z_2$ in the disc $B(0,R),$ the inverse map $g\coloneqq f^{-1}$ satisfies 
$$
|g(z_1))-g(z_2)|^2\;\leq \;\; \frac {C(R,R')}{\log (e+1/|z_1-z_1|)}\int_{B(0,R')} K(z) dz.
$$
\end{lemma}
\noindent 
The original version assumes that $\mu$ is supported
in $\D$, but the more general statement follows  again by scaling. Now \eqref{eq:unif_exp_bound} entails that in our case
$\int_{B(0,R)} K_j(z) dz$ is uniformly bounded, and we obtain a (locally) uniform  modulus of continuity for the inverse maps $f^{-1}_j$. 

Now Theorem \ref{degeco:2.1} follows quickly. Almost surely, 
we have local uniform equicontinuity  for both sequences 
$(f_j)$ and $(f^{-1}_j)$, uniform boundedness of $f_j(z)$ at every point $z$ outside $3\D,$
and thus locally uniform subsequential convergence to a homeomorphism by Arzela-Ascoli.

Moreover, as before in the proof of Lemma \ref{lisaco:2.1},
almost surely each term in the series
\eqref{neumann2} converges locally uniformly. Also,  \eqref{sarja}
implies that the $VMO$-norm of the remainder in \eqref{neumann2} converges uniformly to zero (\cite{AIM}, Theorem 4.3.9). Put together, we deduce the convergence in $VMO(3\D)$ of the whole sequence $f_j$. Since the $f_j$ are analytic outside $2\D,$ this implies the uniqueness of the subsequential limit in $\C$ and establishes  almost sure locally uniform converge 
$f_j\to f_\infty$, where the limit $f_\infty$ is  a self-homeomorphism
of the plane given by  \eqref{neumann3}.
\end{proof}

Let us finally observe that the above proof actually yields the following more general results, stated both for the deterministic and
random homogenization problem. 
\begin{theorem}\label{th:last_deterministic} Let $\mu = \mu_\delta$ be a compactly supported deterministic multiscale function such that for every $0 < \delta < 1$, we have $|\mu_\delta(x)| < 1$ for almost all $x$, and furthermore the dilatation
$K_{\mu_\delta}(x) \coloneqq  \frac{1 + |\mu_\delta(x)|}{1 - |\mu_\delta(x)|}$
is such that  $\int_{\C} e^{pK_{\mu_\delta}}$ is bounded uniformly in $\delta$ for some $p>2.$  Then the associated normalized solutions $F_\delta$ with dilatation $\mu_\delta$ converge locally uniformly in distribution to a homeomorphism $F_\infty \colon \C\to\C$ as $\delta \to 0$.
\end{theorem}

\begin{theorem}\label{th:last_random}   Let $\mu = \mu_\delta$ be a stochastic multiscale function such that for 
$\delta>0$ we have almost surely $|\mu_{\delta}(x)| < 1$ for almost all $x$, and furthermore   for some $p>2$ almost  surely
$\sup_{j\geq1}\int_{\C} e^{pK_{\mu_{2^{-j}}}}<\infty $. 
Then the associated normalized solutions $F_{\mu_{2^{-j}}}$ are almost surely locally uniformly convergent as $j\to \infty$.  
\end{theorem}

\bigskip

\noindent {\bf AFFILIATIONS:}

\begin{description}
\item Kari Astala:\; {\tt kari.astala@aalto.fi}\quad
Department of Mathematics and Systems Analysis, Aalto University, FI-00014 Finland.

\item Steffen Rohde:\quad {\tt rohde@math.washington.edu}\quad
Department of Mathematics, U. Washington, Seattle WA 98195.

\item Eero Saksman:\quad{\tt eero.saksman@helsinki.fi}\quad
Department of Mathematics and Statistics, University of Helsinki, FI-00014 Finland.

\item Terence Tao:\quad {\tt tao@math.ucla.edu}\quad
Department of Mathematics, UCLA, 405 Hilgard Avenue, Los Angeles CA 90095.

\end{description}
\end{document}